\documentclass[reqno,english]{amsart}
\usepackage[T1]{fontenc}
\usepackage[latin9]{inputenc}
\usepackage{babel}
\usepackage{array}
\usepackage{refstyle}
\usepackage{float}
\usepackage{mathrsfs}
\usepackage{mathtools}
\usepackage{enumitem}
\usepackage{amstext}
\usepackage{amsthm}
\usepackage{amssymb}
\usepackage{stmaryrd}
\usepackage{graphicx}
\usepackage[all]{xy}
\usepackage[unicode=true,pdfusetitle,
 bookmarks=true,bookmarksnumbered=false,bookmarksopen=false,
 breaklinks=false,pdfborder={0 0 1},backref=false,colorlinks=false]
 {hyperref}
\hypersetup{
 colorlinks=true,citecolor=blue,linkcolor=blue,linktocpage=true}

\makeatletter

\AtBeginDocument{\providecommand\thmref[1]{\ref{thm:#1}}}
\AtBeginDocument{\providecommand\corref[1]{\ref{cor:#1}}}
\AtBeginDocument{\providecommand\defref[1]{\ref{def:#1}}}
\AtBeginDocument{\providecommand\figref[1]{\ref{fig:#1}}}
\AtBeginDocument{\providecommand\secref[1]{\ref{sec:#1}}}
\AtBeginDocument{\providecommand\propref[1]{\ref{prop:#1}}}
\AtBeginDocument{\providecommand\lemref[1]{\ref{lem:#1}}}
\AtBeginDocument{\providecommand\queref[1]{\ref{que:#1}}}
\AtBeginDocument{\providecommand\remref[1]{\ref{rem:#1}}}
\providecommand{\tabularnewline}{\\}
\RS@ifundefined{subsecref}
  {\newref{subsec}{name = \RSsectxt}}
  {}
\RS@ifundefined{thmref}
  {\def\RSthmtxt{theorem~}\newref{thm}{name = \RSthmtxt}}
  {}
\RS@ifundefined{lemref}
  {\def\RSlemtxt{lemma~}\newref{lem}{name = \RSlemtxt}}
  {}

\numberwithin{equation}{section}
\numberwithin{figure}{section}
\numberwithin{table}{section}
\theoremstyle{plain}
\newtheorem{thm}{\protect\theoremname}[section]
\theoremstyle{remark}
\newtheorem*{notation*}{\protect\notationname}
\theoremstyle{plain}
\newtheorem{lem}[thm]{\protect\lemmaname}
\theoremstyle{definition}
\newtheorem{defn}[thm]{\protect\definitionname}
\theoremstyle{definition}
\newtheorem{example}[thm]{\protect\examplename}
\theoremstyle{remark}
\newtheorem{rem}[thm]{\protect\remarkname}
\theoremstyle{plain}
\newtheorem{cor}[thm]{\protect\corollaryname}
\theoremstyle{plain}
\newtheorem{prop}[thm]{\protect\propositionname}
\theoremstyle{plain}
\newtheorem{question}[thm]{\protect\questionname}
\theoremstyle{plain}
\newtheorem{conjecture}[thm]{\protect\conjecturename}
\theoremstyle{remark}
\newtheorem*{acknowledgement*}{\protect\acknowledgementname}

\@ifundefined{date}{}{\date{}}
\allowdisplaybreaks
\usepackage{needspace}
\usepackage{refstyle}
\usepackage{enumitem}
\usepackage{bbm}

\newref{lem}{refcmd={Lemma \ref{#1}}}
\newref{thm}{refcmd={Theorem \ref{#1}}}
\newref{cor}{refcmd={Corollary \ref{#1}}}
\newref{sec}{refcmd={Section \ref{#1}}}
\newref{sub}{refcmd={Section \ref{#1}}}
\newref{subsec}{refcmd={Section \ref{#1}}}
\newref{chap}{refcmd={Chapter \ref{#1}}}
\newref{prop}{refcmd={Proposition \ref{#1}}}
\newref{exa}{refcmd={Example \ref{#1}}}
\newref{tab}{refcmd={Table \ref{#1}}}
\newref{rem}{refcmd={Remark \ref{#1}}}
\newref{def}{refcmd={Definition \ref{#1}}}
\newref{fig}{refcmd={Figure \ref{#1}}}
\newref{que}{refcmd={Question \ref{#1}}}

\setlist[enumerate]{itemsep=5pt,topsep=3pt}
\setlist[enumerate,1]{label=\textup{(}\roman*\textup{)},ref=\roman*}
\setlist[enumerate,2]{label=(\alph*),ref=\theenumi \alph*}

\makeatletter
\def\subsection{\@startsection{subsection}{2}%
  \z@{.5\linespacing\@plus.7\linespacing}{.1\linespacing}%
  {\normalfont\bfseries}}
\def\subsubsection{\@startsection{subsubsection}{3}%
  \z@{.5\linespacing\@plus.7\linespacing}{.1\linespacing}%
  {\normalfont\bfseries}}
\makeatother
\usepackage{etoolbox}
\makeatletter
\patchcmd{\section}{\scshape}{\bf}{}{}
\makeatother

\usepackage{tikz}
\usetikzlibrary{decorations.pathreplacing}

\makeatletter
\newcommand{\xyR}[1]{%
\makeatletter
\xydef@\xymatrixrowsep@{#1}
\makeatother
} 

\DeclareMathOperator*{\argmin}{\textit{argmin}}

\AtBeginDocument{
  
}

\makeatother

\providecommand{\acknowledgementname}{Acknowledgement}
\providecommand{\conjecturename}{Conjecture}
\providecommand{\corollaryname}{Corollary}
\providecommand{\definitionname}{Definition}
\providecommand{\examplename}{Example}
\providecommand{\lemmaname}{Lemma}
\providecommand{\notationname}{Notation}
\providecommand{\propositionname}{Proposition}
\providecommand{\questionname}{Question}
\providecommand{\remarkname}{Remark}
\providecommand{\theoremname}{Theorem}

\begin{document}
\title[]{A Kaczmarz algorithm for sequences of projections, infinite products,
and applications to frames in IFS $L^{2}$ spaces}
\author{Palle Jorgensen}
\address{(Palle E.T. Jorgensen) Department of Mathematics, The University of
Iowa, Iowa City, IA 52242-1419, U.S.A.}
\email{palle-jorgensen@uiowa.edu}
\urladdr{http://www.math.uiowa.edu/\textasciitilde jorgen/}
\author{Myung-Sin Song}
\address{(Myung-Sin Song) Department of Mathematics, Southern Illinois University
Edwardsville, Edwardsville, IL 62026, U.S.A.}
\email{msong@siue.edu}
\urladdr{http://www.siue.edu/\textasciitilde msong/}
\author{Feng Tian}
\address{(Feng Tian) Department of Mathematics, Hampton University, Hampton,
VA 23668, U.S.A.}
\email{feng.tian@hamptonu.edu}
\begin{abstract}
We show that an idea, originating initially with a fundamental recursive
iteration scheme (usually referred as \textquotedblleft the\textquotedblright{}
Kaczmarz algorithm), admits important applications in such infinite-dimensional,
and non-commutative, settings as are central to spectral theory of
operators in Hilbert space, to optimization, to large sparse systems,
to iterated function systems (IFS), and to fractal harmonic analysis.
We present a new recursive iteration scheme involving as input a prescribed
sequence of selfadjoint projections. Applications include random Kaczmarz
recursions, their limits, and their error-estimates.
\end{abstract}

\subjclass[2000]{Primary 47L60, 46N30, 46N50, 42C15, 65R10, 05C50, 05C75, 31C20, 60J20,
26E40, 65D15, 41A65; Secondary 46N20, 22E70, 31A15, 58J65, 81S25,
68T05.}
\keywords{Hilbert space, Kaczmarz algorithm, randomized Kaczmarz algorithm,
sequences of projections in Hilbert space, convergence, infinite products,
frames, analysis/synthesis, interpolation, optimization, overdetermined
linear systems, transform, feature space, iterated function system,
fractal, Sierpinski gasket, harmonic analysis, approximation, infinite-dimensional
analysis, integral decomposition, random variables, strong operator
topology.}

\maketitle
\tableofcontents{}

\section{\label{sec:Intro}Introduction}

In this paper, we consider certain infinite products of projections.
Our framework is motivated by problems in approximation theory, in
harmonic analysis, in frame theory, and the context of the classical
Kaczmarz algorithm \cite{K-1937}. Traditionally, the infinite-dimensional
Kaczmarz algorithm is stated for sequences of vectors in a specified
Hilbert space $\mathscr{H}$, (typically, $\mathscr{H}$ is an $L^{2}$-space.)
We shall here formulate it instead for sequences of projections. As
a corollary, we get explicit and algorithmic criteria for convergence
of certain \emph{infinite products of projections} in $\mathscr{H}$.
\begin{flushleft}
\textbf{\emph{Organization and main results.}}
\par\end{flushleft}

Our first two sections outline a certain frame-harmonic analysis.
This is the immediate focus of our present applications, but our main
results, dealing with general projection valued Kaczmarz algorithms,
we believe, are of independent interest. They include \thmref{kac}
(products of projections,) and its related results, Corollaries \ref{cor:fi},
\ref{cor:gf}, \ref{cor:kg}, and \ref{cor:ke}. The connection between
infinite products of projections, on the one hand, and more classical
Kaczmarz recursions (for frames), on the other, is spelled out in
Corollaries \ref{cor:ke} and \ref{cor:ke2}. Our main result for
\emph{random} Kaczmarz algorithms is \thmref{rkac}, combined with
\corref{rkac}. In the remaining three sections, we return to applications,
iterated function system, fractals, and random power series.

Our extension of the Kaczmarz algorithm to sequences of projections
is highly nontrivial: While in general convergence questions for infinite
products of projections (in Hilbert space) is difficult (see e.g.,
\cite{MR0051437,MR647807,MR2129258,MR3796644}), we show that our
projection-valued formulation of Kaczmarz' algorithm yields an answer
to this convergence question; as well as a number of applications
to stochastic analysis, and to frame-approximation questions in the
Hilbert space $L^{2}\left(\mu\right)$, where $\mu$ is in a class
of \emph{iterated function system} (IFS) measures (see \cite{MR625600,MR1656855,MR2319756,2016arXiv160308852H,MR3800275}).
The latter refers to a precise multivariable setting, and the class
of measures $\mu$ we consider are fractal measures. (The notion of
\textquotedblleft fractal\textquotedblright{} is defined here relative
to the rank $d$ of the ambient Euclidean space $\mathbb{R}^{d}$
for the particular IFS measure $\mu$ under consideration.) Indeed,
our measures $\mu$ will be singular relative to the Lebesgue measure
on $\mathbb{R}^{d}$. In addition to singularity questions for $\mu$
itself, one must also consider properties of the marginal measures
for $\mu$, and the corresponding slice-direct integral decompositions.
Our first two applications will be the IFS-measures for the Sierpinski
gasket and the Sierpinski carpet, so $d=2$.

In the next section, we introduce this family of measures $\mu$,
called slice-singular measures. We then turn to our Kaczmarz algorithm
for sequences of projections, and its applications.

\section{\label{sec:SSM}Slice-singular measures}

The purpose of the current paper is to perform a systematic analysis
of fractal measures embedded in higher dimensions $d$, such as Sierpinski
triangles ($d=2$), and higher dimensional analogues, $d>2$. The
analysis for $d=1$ begins with the following variant of the F\&M
Riesz theorem:

Consider a choice of period interval, $\left[0,1\right]$, or $\left[-\pi,\pi\right]$,
a positive finite measure $\mu$ with support in the chosen period
interval; and the usual Fourier frequencies realized as complex exponentials
$e_{n}$, $n\in\mathbb{Z}$. Set $\mathbb{N}_{0}=\left\{ 0\right\} \cup\mathbb{N}$.
\begin{thm}[F\&M Riesz]
\label{thm:FM}The subset $\left\{ e_{n}\mid n\in\mathbb{N}_{0}\right\} $
is total in $L^{2}\left(\mu\right)$ if and only if $\mu$ is singular
with respect to Lebesgue measure.
\end{thm}

The corresponding result is false when $d>1$, and the question is:
What is a natural extension of F\&M Riesz' theorem to higher dimensions,
modeling the above formulation? One of the motivations for this is
a certain construction of frame algorithms in $L^{2}\left(\mu\right)$;
in the form started for $d=1$ in \cite{MR2319756,2016arXiv160308852H,MR3796641,HERR2018}.
For general frame theory, including projection valued frames, see
e.g., \cite{MR2147063,MR2367342,MR2362796,MR3423689,MR3611473,MR3817340,MR3894265,MR3906284,MR3910931,MR3778680,MR3896128}.

\thmref{FM} does \emph{not} extend to 2D, or higher dimensions. In
1D, the standard F\&M Riesz theorem is used at a crucial point; but
there is \emph{not} a direct extension of the theorem in one variable.
To get a harmonic analysis of $L^{2}\left(\mu\right)$, with $supp\left(\mu\right)\subset\mathbb{R}^{d}$,
$d\geq2$, one must assume instead that $\mu$ is \emph{slice singular};
see \defref{SS}. It is possible to view the result as an extension
of F\&M Riesz' theorem to higher dimensions.

For the sake of stressing the idea, we shall consider the case $d=2$
in most detail.
\begin{notation*}
Let $\left(X,\mathscr{F}\right)$ be a measurable space. $\mathcal{M}\left(X\right)$
denotes all Borel measures on $\mathscr{F}$. The set $\mathcal{M}^{+}\left(X\right)$
consists of all positive measures in $\mathcal{M}\left(X\right)$,
and $\mathcal{M}_{1}^{+}\left(X\right)$ the subset of probability
measures. We shall also use standard multi-index notations.
\end{notation*}
Let $\left(X\times Y,\mathscr{B}_{X}\times\mathscr{B}_{Y},\mu\right)$
be a measure space, where $X$, $Y$ are equipped with $\sigma$-algebras
$\mathscr{B}_{X}$, $\mathscr{B}_{Y}$ respectively, and $\mu$ is
defined on the \emph{product $\sigma$-algebra}.
\begin{lem}[Disintegration]
Every positive measure $\mu$ on $X\times Y$ w.r.t. the product
$\sigma$-algebra yields a unique representation as follows:
\begin{enumerate}
\item $\xi:=\mu\circ\pi_{X}^{-1}$ is a measure on $\left(X,\mathscr{B}_{X}\right)$;
\item \label{enu:dis21}There exists a conditional measure $\sigma^{x}\left(dy\right):=\sigma\left(x,dy\right)$
on $\left(Y,\mathscr{B}_{Y}\right)$, defined for a.a. $x\in X$,
such that 
\begin{equation}
d\mu=\int\sigma^{x}\left(dy\right)d\xi\left(x\right).\label{eq:A1}
\end{equation}
\end{enumerate}
\end{lem}

The precise meaning of (\ref{eq:A1}) is as follows: For all measurable
functions $F$ on $X\times Y$, we have 
\begin{equation}
\iint_{X\times Y}Fd\mu=\int_{X}\left(\int_{Y}F\left(x,y\right)\sigma^{x}\left(dy\right)\right)d\xi\left(x\right).\label{eq:A2}
\end{equation}
The decomposition (\ref{eq:A2}) is often referred to as a \emph{Rohlin
disintegration formula}.
\begin{defn}
\label{def:SS}A Borel measure $\mu$ on $J^{2}:=\left[0,1\right]\times\left[0,1\right]$
is called \emph{slice singular} iff (Def.)
\begin{enumerate}
\item $\xi=\mu\circ\pi_{1}^{-1}$ is singular; and
\item for a.a. $x$ w.r.t. $\xi$, the measure $\sigma^{x}\left(\cdot\right)$
is singular.
\end{enumerate}
``Singular'' is defined relative to Lebesgue measure.
\end{defn}

\begin{thm}
\label{thm:SM}If $\mu$ is slice singular on $J^{2}$, then $\left\{ e_{n}\right\} _{n\in\mathbb{N}_{0}^{2}}$
has dense span in $L^{2}\left(\mu\right)$, where $e_{n}\left(x\right)=e^{i2\pi\left(n_{1}x_{1}+n_{2}x_{2}\right)}$,
for all $n=\left(n_{1},n_{2}\right)\in\mathbb{N}_{0}^{2}$, and $x=\left(x_{1},x_{2}\right)\in J^{2}$.
\end{thm}

\begin{proof}
We shall show that, if $\left\langle F,e_{n}\right\rangle _{L^{2}\left(\mu\right)}=0$,
$\forall n\in\mathbb{N}_{0}^{2}$, then $F=0$ $\mu$-a.e. But 
\begin{align}
\left\langle F,e_{n}\right\rangle _{L^{2}\left(\mu\right)} & =\int_{0}^{1}e_{n_{1}}\left(x\right)\left(\int_{0}^{1}e_{n_{2}}\left(y\right)\overline{F\left(x,y\right)}\sigma^{x}\left(dy\right)\right)d\xi\left(x\right)\nonumber \\
 & =0,\;\forall n=\left(n_{1},n_{2}\right)\in\mathbb{N}_{0}^{2}\nonumber \\
 & \Downarrow\quad\left(\text{since \ensuremath{\xi} is singular}\right)\label{eq:fm1}\\
\int_{0}^{1}e_{n_{2}}\left(y\right) & \overline{F\left(x,y\right)}\sigma^{x}\left(dy\right)=0,\;a.a.\;x,\;\forall n_{2}\in\mathbb{N}_{0}\nonumber \\
 & \Downarrow\quad\left(\text{since \ensuremath{\sigma^{x}\left(\cdot\right)} is singular a.a. \ensuremath{x}}\right)\label{eq:fm2}\\
F\left(x,y\right) & =0,\;a.a.\;\left(x,y\right)\;\text{w.r.t. }\mu.\nonumber 
\end{align}

This gives the desired conclusion that $\left\{ e_{n}\right\} _{n\in\mathbb{N}_{0}^{2}}$
is total in $L^{2}\left(\mu\right)$.
\end{proof}
\begin{example}[$d=2$]
 $\mu\in\mathcal{M}^{+}\left(\mathbb{T}^{2}\right)$, $W=$ Sierpinski
gasket/carpet (\figref{2df}).

Note that, for a.a. $x$ w.r.t. $\xi$, the measure $\sigma^{x}$
on $A\left(x\right)=\left\{ y\mid\left(x,y\right)\in W\right\} $
is a fractal measure with variable gap size; and by Kakutani's theorem,
for a.a. $x$, $\sigma^{x}\left(dy\right)$ is singular relative to
the Lebesgue measure. Hence we can apply F\&M Riesz as in (\ref{eq:fm1}),
and (\ref{eq:fm2}).
\end{example}

The detailed properties of the fractals from \figref{2df} (A) and
(B) will be derived in \secref{Sie} below.

\begin{figure}[H]
\begin{tabular}{ccccc}
\includegraphics[width=0.25\columnwidth]{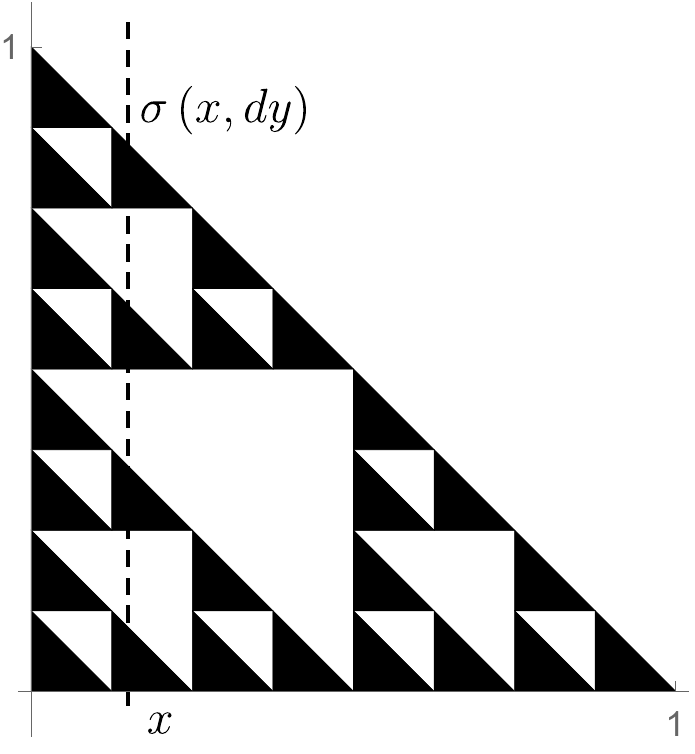} &  &  &  & \includegraphics[width=0.25\columnwidth]{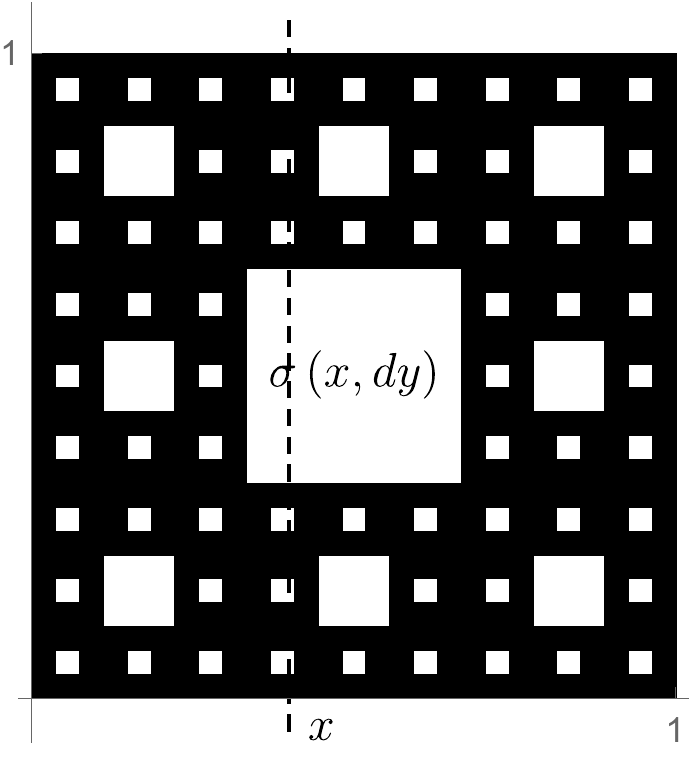}\tabularnewline
(A) Sierpinski gasket &  &  &  & (B) Sierpinski carpet\tabularnewline
\end{tabular}

\caption{\label{fig:2df}Examples of slice singular measures.}
\end{figure}

While the Sierpinski constructions in \figref{2df} are better known
as self-affine planar sets, it is in fact the corresponding \emph{measures}
which are important for algorithms and for frame-harmonic analysis.
As it turns out, the particular affine maps (see (\ref{eq:F1}), (\ref{eq:F2}),
and \figref{isp} below) going into the Sierpinski constructions are
in fact special cases of a more general family of iterated function
systems (IFS.) They are discussed in detail in sections \ref{sec:gifs}
and \ref{sec:Sie} below. Brief preview: Given a system of contractive
mappings, affine or conformal, there are then two associated fixed-point
problems, one for compact sets, and the other for probability measures:
The case of the sets $W$ is discussed in (\ref{eq:d9}), and the
measures $\mu$ in (\ref{eq:d7}). For a fixed IFS, the set in question
arises as the support of an associated IFS-measure $\mu$. Probabilistic
features of these constructions are outlined in sect \ref{sec:gifs},
and their fractal properties, in sect \ref{sec:Sie}, below. In particular,
we show that these planar Sierpinski measures $\mu$ are slice-singular.

\section{Frames, projections, and Kaczmarz algorithms}

While earlier approaches to the Kaczmarz algorithm in Hilbert space
have dealt with recursive constructions of vectors, as needed in optimization
problems, or in harmonic analysis, we present here an extension of
the algorithm to the context of countable systems of \emph{selfadjoint
projections} in a Hilbert space. As outlined in subsequent sections
of our paper, the projection setting is motivated directly by applications;
the \emph{randomized} Kaczmarz algorithms, just one of them.

For the benefit of readers, and for later reference, we include below
a brief review of fundamentals for the classical Kaczmarz algorithm,
and its variants. This also gives us a suitable framework for our
present results: An operator theoretic extension of Kaczmarz, with
applications to multivariable fractal measures.

\emph{Literature guide}: In addition to Kaczmarz' pioneering paper
\cite{K-1937}, there are also the following more recent developments
of relevance to our present discussion \cite{MR1898684,MR2208766,MR2140451,MR2263965,MR2311862,MR2721177,MR2835851,MR3117886,MR3159297,MR3450541,MR3439812,MR3796634,MR3846956,MR3896982},
as well as \cite{2016arXiv160308852H,MR3796641,HERR2018}.\\

The classical Kaczmarz algorithm is an iterative method for solving
systems of linear equations, for example, $Ax=b$, where $A$ is an
$m\times n$ matrix.

Assume the system is consistent. Let $x_{0}$ be an arbitrary vector
in $\text{\ensuremath{\mathbb{R}^{n}}}$, and set 
\begin{equation}
x_{k}:=\argmin_{\left\langle a_{j},x\right\rangle =b_{j}}\left\Vert x-x_{k-1}\right\Vert ^{2},\;k\in\mathbb{N};\label{eq:C1}
\end{equation}
where $j=k\mod m$, and $a_{j}$ denotes the $j^{th}$ row of $A$.
At each iteration, the minimizer is given by 
\begin{equation}
x_{k}=x_{k-1}+\frac{b_{j}-\left\langle a_{j},x_{k-1}\right\rangle }{\left\Vert a_{j}\right\Vert ^{2}}a_{j}.\label{eq:C2}
\end{equation}
That is, the algorithm recursively projects the current state onto
the hyperplane determined by the next row vector of $A$.

There is a stochastic version of (\ref{eq:C2}), where the row vectors
of $A$ are selected randomly \cite{MR2500924}. Also see Sections
\ref{subsec:rkac} and \ref{subsec:axy} below.
\begin{rem}
Following standard conventions in approximation theory, we use the
notation \emph{argmin} for denoting the vector which realizes a specified
optimization; in this case (see \figref{kap}), we refer to the minimum
problem on the right hand side in eq (\ref{eq:C1}). So in the particular
instance of the Kaczmarz algorithm (\ref{eq:C2}), we are in finite
dimensions, and there is then an easy, geometric, and explicit formula
for the argmin vector occurring in each step of the algorithm, see
\figref{kap}.
\end{rem}

The Kaczmarz algorithm can be formulated in the Hilbert space setting
as follows:
\begin{defn}
Let $\left\{ e_{j}\right\} _{j\in\mathbb{N}_{0}}$ be a spanning set
of unit vectors in a Hilbert space $\mathscr{H}$, i.e., $span\left\{ e_{j}\right\} $
is dense in $\mathscr{H}$. For all $x\in\mathscr{H}$, let $x_{0}=e_{0}$,
and set 
\begin{equation}
x_{k}:=x_{k-1}+e_{k}\left\langle e_{k},x-x_{k-1}\right\rangle .\label{eq:C3}
\end{equation}
We say the sequence $\left\{ e_{j}\right\} _{j\in\mathbb{N}_{0}}$
is \emph{effective }if $\left\Vert x_{k}-x\right\Vert \rightarrow0$
as $k\rightarrow\infty$, for all $x\in\mathscr{H}$.
\end{defn}

\begin{rem}
A key motivation for our present analysis is an important result by
Stanis\l aw Kwapie\'{n}  and Jan Mycielski \cite{MR2263965}, giving
a criterion for stationary sequences (referring to a suitable $L^{2}\left(\mu\right)$)
to be effective.
\end{rem}

\textbf{Observation.} Equation (\ref{eq:C3}) yields, by forward induction:
\begin{eqnarray*}
x-x_{k} & = & \left(1-P_{k}\right)\left(x-x_{k-1}\right)\\
 & = & \left(1-P_{k}\right)\left(1-P_{k-1}\right)\left(x-x_{k-2}\right)\\
 & \vdots\\
 & = & \left(1-P_{k}\right)\left(1-P_{k-1}\right)\cdots\left(1-P_{0}\right)x,
\end{eqnarray*}
where $P_{j}$ is the orthogonal projection onto $e_{j}$.

\begin{figure}[H]
\begin{tabular}{>{\centering}p{0.45\columnwidth}>{\centering}p{0.45\columnwidth}}
\includegraphics[width=0.35\columnwidth]{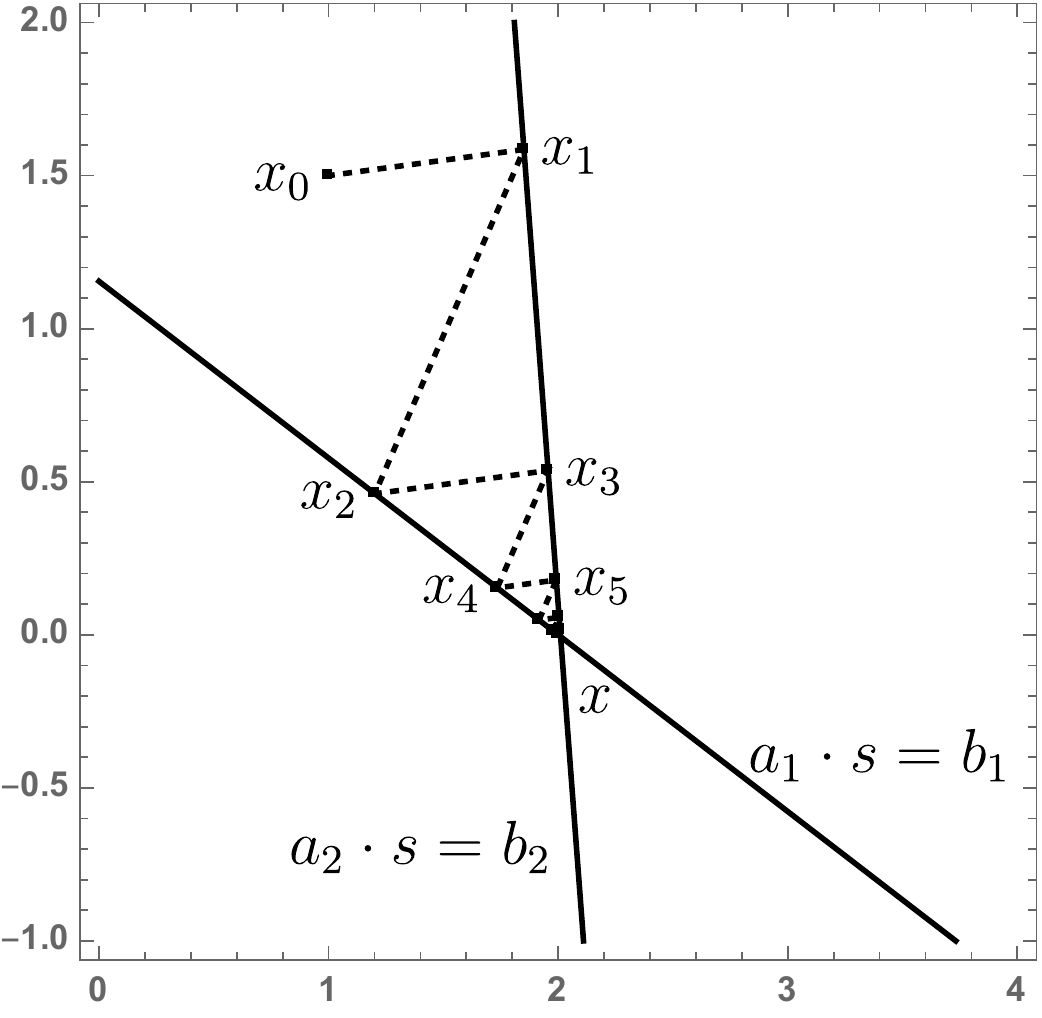} & \includegraphics[width=0.35\columnwidth]{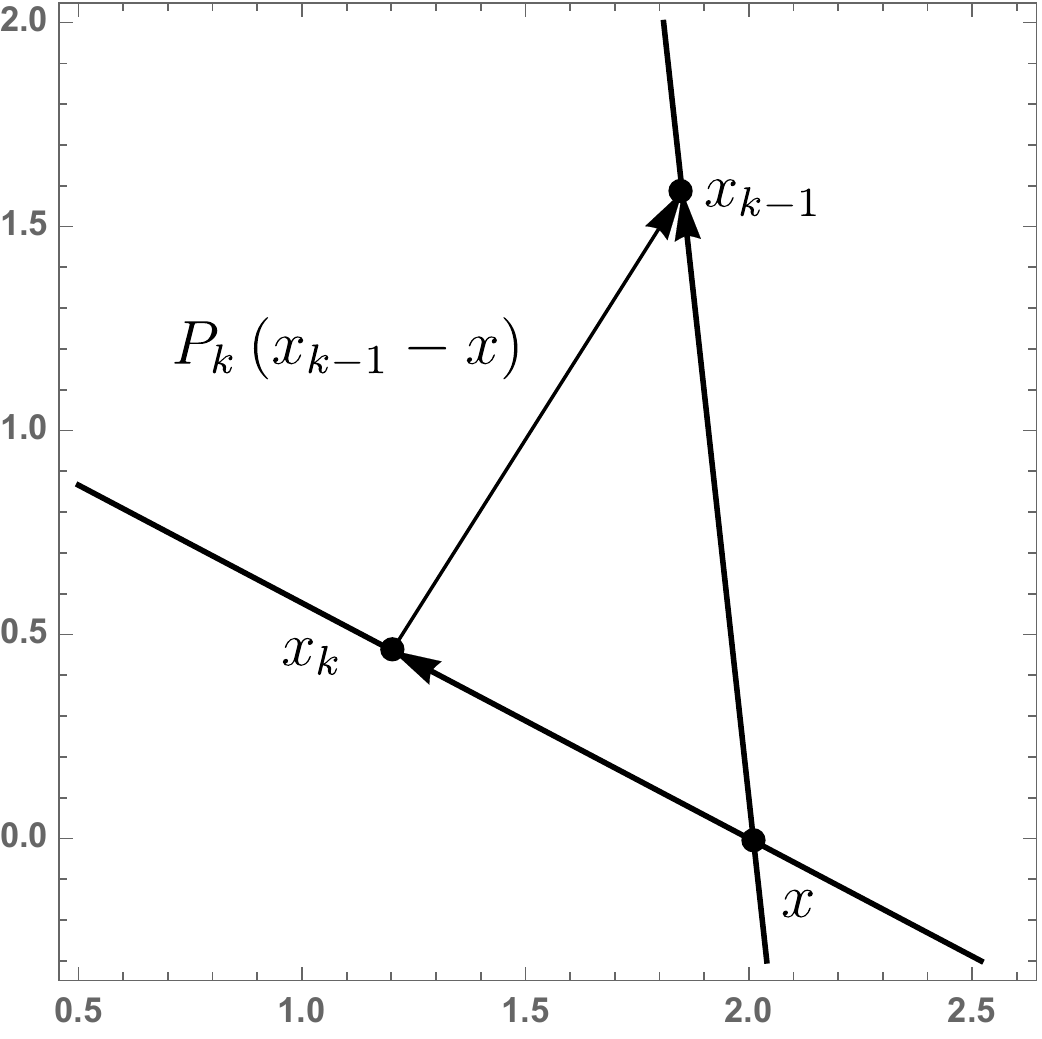}\tabularnewline
(A) Approximate solution; random starting point $x_{0}$ & (B) orthogonality relation

$\left\Vert x_{k-1}-x\right\Vert ^{2}=\left\Vert x_{k-1}-x_{k}\right\Vert ^{2}+\left\Vert x_{k}-x\right\Vert ^{2}$\tabularnewline
\end{tabular}

\hfill{}

\caption{\label{fig:kap} Solution to $Ax=b$ by the Kaczmarz algorithm, with
$a_{1}=\left(\cos\left(\pi/3\right),\sin\left(\pi/3\right)\right)$,
$a_{2}=\left(\cos\left(0.1\right),\sin\left(0.1\right)\right)$, and
$b=\left(1,2\right)$.}
\end{figure}

\subsection{Algorithms, and products of projections}

We now present an extension of the Kaczmarz algorithm; an extension
to a setting of an infinite sequence of selfadjoint projections, as
opposed to the classical case of sequences of vectors in Hilbert space.
There are more general results on limits of iterated products of selfadjoint
projections. See \cite{MR3796644} and also \cite{MR0051437,MR647807,MR2129258}.
For applications of infinite products of operators to central problems
in mathematical physics, see e.g., papers by D. Ruelle et al \cite{MR0284067,MR534172,MR647807}.
\begin{flushleft}
\textbf{Preliminaries}
\par\end{flushleft}

Let $\mathscr{H}$ be a Hilbert space. An operator $P:\mathscr{H}\rightarrow\mathscr{H}$
is said to be a \emph{selfadjoint projection} iff (Def.) $P=P^{*}=P^{2}$.
It is known that there is a bijective correspondence between:
\begin{enumerate}
\item all closed subspaces $\mathscr{M}\subset\mathscr{H}$; and
\item the set of all selfadjoint projections $P$.
\end{enumerate}
If $\mathscr{M}$ is as in (i), then $P$ may be obtained from the
axioms for $\mathscr{H}$; and we have 
\begin{equation}
P\mathscr{H}=\mathscr{M}=\left\{ x\in\mathscr{H}\mathrel{;}Px=x\right\} .\label{eq:cp1}
\end{equation}
Conversely, if $P$ is given as in (ii), then $\mathscr{M}$ (see
(\ref{eq:cp1})) is a closed subspace in $\mathscr{H}$.

The ortho-complement
\begin{equation}
\mathscr{M}^{\perp}:=\mathscr{H}\ominus\mathscr{M}=\left\{ x\in\mathscr{H}\mathrel{;}Px=0\right\} 
\end{equation}
is the closed subspace corresponding to the selfadjoint projection
$P^{\perp}:=1-P$. (Here, we denote the identity operator in $\mathscr{H}$
by $1$, as it is the unit in the $C^{*}$-algebra $\mathscr{B}\left(\mathscr{H}\right)$.)
\begin{rem}
For our present purpose, all projections will be assumed selfadjoint.
On occasion, to save space, we shall simply say \textquotedblleft projection\textquotedblright{}
when selfadjointness is implicit. (We note that selfadjoint projections
yield orthogonal sum-splittings, and are therefore often, equivalently,
referred to as orthogonal projections.)
\end{rem}

We shall further make use of the \emph{lattice operations} corresponding
to the correspondence (i)$\leftrightarrow$(ii) above:

If $\mathscr{M}_{i}$, $i=1,2$, are closed subspaces with corresponding
projections $P_{i}$, $i=1,2$; then TFAE:
\begin{align}
\mathscr{M}_{1} & \subseteq\mathscr{M}_{2},\;\text{and}\\
P_{1} & =P_{1}P_{2}.
\end{align}
Moreover, for a pair of  projections $\left\{ P_{i}\right\} _{i=1,2}$,
TFAE:
\begin{eqnarray*}
P_{1} & = & P_{1}P_{2}\\
 & \Updownarrow\\
P_{1} & = & P_{2}P_{1}\\
 & \Updownarrow\\
\left\Vert P_{1}x\right\Vert  & \leq & \left\Vert P_{2}x\right\Vert ,\;\forall x\in\mathscr{H}\\
 & \Updownarrow\\
\left\langle x,P_{1}x\right\rangle  & \leq & \left\langle x,P_{2}x\right\rangle ,\;\forall x\in\mathscr{H}.
\end{eqnarray*}

\textbf{Caution: }In general, the class of selfadjoint projections
is \emph{not} closed under products, under sums, or under differences.
\begin{thm}
\label{thm:kac}Let $\left\{ P_{j}\right\} _{j\in\mathbb{N}_{0}}$
be a system of selfadjoint projections in a Hilbert space $\mathscr{H}$.
For all $n\in\mathbb{N}_{0}$, set 
\begin{align}
T_{n} & =\left(1-P_{n}\right)\left(1-P_{n-1}\right)\cdots\left(1-P_{0}\right),\;\text{and}\label{eq:f8}\\
Q_{n} & =P_{n}\left(1-P_{n-1}\right)\cdots\left(1-P_{0}\right),\quad Q_{0}=P_{0}.\label{eq:f2}
\end{align}
Then, 
\begin{align}
1-T_{n}^{*}T_{n} & =\sum_{j=0}^{n}Q_{j}^{*}Q_{j},\;\text{and}\label{eq:cc6}\\
1-T_{n} & =\sum_{j=0}^{n}Q_{j}.\label{eq:cc7}
\end{align}
\end{thm}

\begin{rem}
The operator products introduced in formulas (\ref{eq:f8}) and (\ref{eq:f2})
above will play an important role in our subsequent considerations.
Hence, when we refer to $Q_{n}$, and $T_{n}$, we shall mean the
particular operator products in (\ref{eq:f8}) and (\ref{eq:f2}).
The input in our algorithm will be a fixed system of selfadjoint projections,
$P_{n}$.

Note that the factors making up the operator products in (\ref{eq:f8})
and (\ref{eq:f2}) are non-commuting. We stress that non-comutativity
is an important (and subtle) feature of the theory of operator frames;
see e.g., \cite{MR3642406}.
\end{rem}

\begin{proof}[Proof of \thmref{kac}]
One checks that $T_{n}=T_{n-1}-Q_{n}$, so that 
\begin{eqnarray*}
T_{n}^{*}T_{n} & = & \left(T_{n-1}^{*}-Q_{n}^{*}\right)\left(T_{n-1}-Q_{n}\right)\\
 & = & T_{n-1}^{*}T_{n-1}-T_{n-1}^{*}Q_{n}-Q_{n}^{*}T_{n-1}+Q_{n}^{*}Q_{n}\\
 & = & T_{n-1}^{*}T_{n-1}-Q_{n}^{*}Q_{n}-Q_{n}^{*}Q_{n}+Q_{n}^{*}Q_{n}\\
 & = & T_{n-1}^{*}T_{n-1}-Q_{n}^{*}Q_{n}\\
 & = & T_{n-2}^{*}T_{n-2}-Q_{n-1}^{*}Q_{n-1}-Q_{n}^{*}Q_{n}\\
 & \vdots\\
 & = & 1-P_{0}-\sum_{j=1}^{n}Q_{j}^{*}Q_{j}\\
 & = & 1-\sum_{j=0}^{n}Q_{j}^{*}Q_{j}.
\end{eqnarray*}
Since $Q_{n}=T_{n-1}-T_{n}$, so 
\begin{align*}
\sum_{j=0}^{n}Q_{j} & =Q_{0}+\left(T_{0}-T_{1}\right)+\left(T_{1}-T_{2}\right)+\cdots+\left(T_{n-1}-T_{n}\right)\\
 & =P_{0}+1-P_{0}-T_{n}=1-T_{n}.
\end{align*}
\end{proof}
Let $\mathscr{H}$ be a Hilbert space, and let $\left\{ A_{n}\right\} _{n\in\mathbb{N}}$
be a sequence of bounded operators in $\mathscr{H}$, i.e., $A_{n}\in\mathscr{B}\left(\mathscr{H}\right)$,
$\forall n\in\mathbb{N}$. We shall need the following two notions
of convergence in $\mathscr{B}\left(\mathscr{H}\right)$.
\begin{defn}
~
\begin{enumerate}
\item We say that $A_{n}\rightarrow0$ in the \emph{strong operator topology}
(SOT) iff (Def.) $\lim_{n\rightarrow\infty}\left\Vert A_{n}x\right\Vert =0$
for all vectors $x\in\mathscr{H}$.
\item We say that $A_{n}\rightarrow0$ in the \emph{weak operator topology}
(WOT) iff (Def.) $\lim_{n\rightarrow\infty}\left\langle x,A_{n}y\right\rangle =0$
for all pairs of vectors $x,y\in\mathscr{H}$. Here $\left\langle \cdot,\cdot\right\rangle $
refers to the inner product in $\mathscr{H}$.
\end{enumerate}
\end{defn}

\begin{cor}
\label{cor:fi}The following are equivalent:
\begin{enumerate}
\item \label{enu:mf1}$1=\sum_{j\in\mathbb{N}_{0}}Q_{j}^{*}Q_{j}$ in the
weak operator topology.
\item $1=\sum_{j\in\mathbb{N}_{0}}Q_{j}$ in the strong operator topology.
\item \label{enu:mf3}$T_{n}\rightarrow0$ in the strong operator topology.
\end{enumerate}
\end{cor}

\begin{rem}
Under suitable conditions on $Q_{n}$ one can show that the convergence
in part (\ref{enu:mf1}) of the corollary also holds in the strong
operator topology.
\end{rem}

\begin{defn}
The system $\left\{ P_{j}\right\} _{j\in\mathbb{N}_{0}}$ is called
\emph{effective} if $T_{n}\rightarrow0$ in the strong operator topology.
\end{defn}

\begin{cor}
\label{cor:gf}Suppose the system $\left\{ P_{j}\right\} _{j\in\mathbb{N}_{0}}$
is effective. Then, for all $x\in\mathscr{H}$, 
\begin{equation}
x=\sum_{j\in\mathbb{N}_{0}}Q_{j}x.\label{eq:C5}
\end{equation}
Moreover, for all $x,y\in\mathscr{H}$, 
\begin{equation}
\left\langle x,y\right\rangle =\sum_{j\in\mathbb{N}_{0}}\left\langle Q_{j}x,Q_{j}y\right\rangle ;
\end{equation}
and in particular, 
\begin{equation}
\left\Vert x\right\Vert ^{2}=\sum_{j\in\mathbb{N}_{0}}\left\Vert Q_{j}x\right\Vert ^{2}.\label{eq:C7}
\end{equation}
\end{cor}

\begin{rem}
The system of operators $\left\{ Q_{j}\right\} _{j\in\mathbb{N}_{0}}$
in \corref{gf} has frame-like properties. Specifically, the mapping
\[
\mathscr{H}\ni x\xmapsto{\;V\;}\left(Q_{j}x\right)\in l^{2}\left(\mathbb{N}_{0}\right)\otimes\mathscr{H}
\]
plays the role of an analysis operator, and the synthesis operator
$V^{*}$ is given by 
\[
l^{2}\left(\mathbb{N}_{0}\right)\otimes\mathscr{H}\ni\xi\xmapsto{\;V^{*}\;}\sum_{j\in\mathbb{N}_{0}}Q_{j}^{*}\xi_{j}.
\]
Note that $1=V^{*}V$, by part (\ref{enu:mf1}) of \corref{fi}; and
eq. (\ref{eq:C7}) is the generalized Parseval identity. Also see
\propref{fd} below.
\end{rem}

\begin{prop}
\label{prop:fd}Let $\left\{ P_{j}\right\} _{j\in\mathbb{N}_{0}}$
be an effective system. Then there exits a Hilbert space $\mathscr{K}$,
an isometry $V:\mathscr{H}\rightarrow\mathscr{K}$, and selfadjoint
projections $E_{j}$ in $\mathscr{K}$, such that $Q_{j}^{*}Q_{j}=V^{*}E_{j}V$,
for all $j\in\mathbb{N}_{0}$. Thus, 
\begin{equation}
1=\sum_{j\in\mathbb{N}_{0}}Q_{j}^{*}Q_{j}=\sum_{j\in\mathbb{N}_{0}}V^{*}E_{j}V.\label{eq:d1}
\end{equation}
\end{prop}

\begin{proof}
Let $\mathscr{K}=l^{2}\left(\mathbb{N}_{0}\right)\otimes\mathscr{H}\left(=\oplus_{\mathbb{N}_{0}}\mathscr{H}\right)$,
and set $V:\mathscr{H}\rightarrow\mathscr{K}$ by 
\[
Vx=\left(Q_{j}x\right)_{j\in\mathbb{N}_{0}}.
\]
Then, for all $x\in\mathscr{H}$ and $y=\left(y_{j}\right)\in\mathscr{K}$,
\[
\left\langle Vx,y\right\rangle _{\mathscr{K}}=\sum\left\langle Q_{j}x,y_{j}\right\rangle _{\mathscr{H}}=\left\langle x,\sum Q_{j}^{*}y_{j}\right\rangle _{\mathscr{H}}.
\]
Hence the adjoint operator $V^{*}$ is given by 
\[
V^{*}y=\sum_{j\in\mathbb{N}_{0}}Q_{j}^{*}y_{j}.
\]

For all $j\in\mathbb{N}_{0}$, let $E_{j}:\mathscr{K}\rightarrow\mathscr{K}$
be the projection, 
\[
E_{j}y=\left(0,\cdots,0,y_{j},0,\cdots\right),\;\forall y=\left(y_{j}\right)\in\mathscr{K}.
\]
Then $Q_{j}^{*}Q_{j}=V^{*}E_{j}V$, and  (\ref{eq:d1}) follows from
this.
\end{proof}
Let $\mathscr{H}$ be a fixed Hilbert space. We shall have occasion
to use Dirac's notation for rank-one operators in $\mathscr{H}$:
If $u,v\in\mathscr{H}$, we set $\left|u\left\rangle \right\langle v\right|$
the operator, which is defined by 
\[
\left|u\left\rangle \right\langle v\right|\left(x\right)=\left\langle v,x\right\rangle _{\mathscr{H}}u;
\]
or in physics terminology, 
\[
\left|u\left\rangle \right\langle v\right|\left.x\right\rangle =\left|u\right\rangle \left\langle v,x\right\rangle _{\mathscr{H}}.
\]
Note the following: For vectors $u_{i},v_{i}$, $=1,2$, we have:
\[
\left(\left|u_{1}\left\rangle \right\langle v_{1}\right|\right)\left(\left|u_{2}\left\rangle \right\langle v_{2}\right|\right)=\left\langle v_{1},u_{2}\right\rangle _{\mathscr{H}}\left|u_{1}\left\rangle \right\langle v_{2}\right|.
\]
For the adjoint operators, we have: 
\[
\left|u\left\rangle \right\langle v\right|^{*}=\left|v\left\rangle \right\langle u\right|.
\]
If $B\in\mathscr{B}\left(\mathscr{H}\right)$, we have 
\[
B\left|u\left\rangle \right\langle v\right|=\left|Bu\left\rangle \right\langle v\right|;
\]
and 
\[
\left|u\left\rangle \right\langle v\right|B=\left|u\left\rangle \right\langle B^{*}v\right|.
\]

\subsection{The case of rank-1 projections in Hilbert space}

Let $\left\{ P_{j}\right\} _{j\in\mathbb{N}_{0}}$ be a system of
rank-1  projections, i.e., $P_{j}=\left|e_{j}\left\rangle \right\langle e_{j}\right|$,
where $\left\{ e_{j}\right\} _{j\in\mathbb{N}_{0}}$ is a set of unit
vectors in $\mathscr{H}$. When the system $\left\{ e_{j}\right\} $
is independent, then the corresponding family of projections $P_{j}=\left|e_{j}\left\rangle \right\langle e_{j}\right|$
is non-commutative.

It follows from (\ref{eq:f2}) that every $Q_{j}$ is a rank-1 operator
with range in $span\left\{ e_{j}\right\} $. Thus there exists a unique
$g_{j}\in\mathscr{H}$ such that 
\begin{equation}
Q_{j}=\left|e_{j}\left\rangle \right\langle g_{j}\right|,\;j\in\mathbb{N}_{0}.\label{eq:C16}
\end{equation}

\begin{lem}
\label{lem:qn}Given $\left\{ P_{j}\right\} _{j\in\mathbb{N}_{0}}$
a sequence of s.a. projections in $\mathscr{H}$; set 
\begin{equation}
Q_{n}:=P_{n}P_{n-1}^{\perp}\cdots P_{1}^{\perp}P_{0}^{\perp},\label{eq:C17}
\end{equation}
where $P_{j}^{\perp}:=1-P_{j}$; then 
\begin{equation}
Q_{n}=P_{n}\left(1-\sum\nolimits _{j=0}^{n-1}Q_{j}\right).
\end{equation}
\end{lem}

\begin{proof}
By definition, we have
\begin{eqnarray*}
Q_{n} & = & P_{n}\left(1-P_{n-1}\right)P_{n-2}^{\perp}\cdots P_{0}^{\perp}\\
 & = & P_{n}P_{n-2}^{\perp}\cdots P_{0}^{\perp}-P_{n}P_{n-1}P_{n-2}^{\perp}\cdots P_{0}^{\perp}\\
 & = & P_{n}P_{n-2}^{\perp}\cdots P_{0}^{\perp}-P_{n}Q_{n-1}\\
 & = & P_{n}P_{n-3}^{\perp}\cdots P_{0}^{\perp}-P_{n}Q_{n-2}-P_{n}Q_{n-1}\\
 & \vdots\\
 & = & P_{n}P_{0}^{\perp}-P_{n}Q_{1}-P_{n}Q_{2}-\cdots-P_{n}Q_{n-1}\\
 & = & P_{n}-\sum\nolimits _{j=0}^{n-1}P_{n}Q_{j}\\
 & = & P_{n}\left(1-\sum\nolimits _{j=0}^{n-1}Q_{j}\right).
\end{eqnarray*}
\end{proof}
\begin{cor}
\label{cor:kg}The vectors $\left\{ g_{j}\right\} $ in (\ref{eq:C16})
are determined recursively by
\begin{align}
g_{0} & =e_{0}\\
g_{n} & =e_{n}-\sum_{j=0}^{n-1}\left\langle e_{j},e_{n}\right\rangle g_{j}.\label{eq:C20}
\end{align}
\end{cor}

\begin{proof}
For all $x\in\mathscr{H}$, it follows from \lemref{qn}, that
\begin{eqnarray*}
Q_{n}x & = & P_{n}x-\sum_{j=0}^{n-1}P_{n}Q_{j}x\\
 & \Updownarrow\\
e_{n}\left\langle g_{n},x\right\rangle  & = & e_{n}\left\langle e_{n},x\right\rangle -\sum_{j=0}^{n-1}e_{n}\left\langle e_{n},e_{j}\right\rangle \left\langle g_{j},x\right\rangle .
\end{eqnarray*}
That is, $g_{n}=e_{n}-\sum_{j=0}^{n-1}\left\langle e_{j},e_{n}\right\rangle g_{j}$.
\end{proof}
\begin{cor}
\label{cor:ke}Assume $\left\{ \left|e_{j}\left\rangle \right\langle e_{j}\right|\right\} _{j\in\mathbb{N}_{0}}$
is effective, and let $Q_{j}=\left|e_{j}\left\rangle \right\langle g_{j}\right|$
be as above. Then, for all $x\in\mathscr{H}$, we have 
\begin{equation}
x=\sum_{j\in\mathbb{N}_{0}}\left\langle g_{j},x\right\rangle e_{j}.\label{eq:C21}
\end{equation}
In particular, for all $A\in\mathscr{B}\left(\mathscr{H}\right)$,
then
\begin{equation}
Ax=\sum_{j\in\mathbb{N}_{0}}\left\langle A^{*}g_{j},x\right\rangle e_{j}.
\end{equation}

Moreover, for all $x,y\in\mathscr{H}$, 
\begin{align*}
\left\langle x,y\right\rangle  & =\sum_{j\in\mathbb{N}_{0}}\left\langle x,g_{j}\right\rangle \left\langle g_{j},y\right\rangle ,\;\text{and}\\
\left\Vert x\right\Vert ^{2} & =\sum_{j\in\mathbb{N}_{0}}\left|\left\langle g_{j},x\right\rangle \right|^{2}.
\end{align*}
\end{cor}

\begin{proof}
By assumption, $Q_{j}^{*}Q_{j}=\left|g_{j}\left\rangle \right\langle g_{j}\right|$,
hence 
\[
\left\langle x,y\right\rangle =\sum_{j\in\mathbb{N}_{0}}\left\langle x,Q_{j}^{*}Q_{j}y\right\rangle =\sum_{j\in\mathbb{N}_{0}}\left\langle x,g_{j}\right\rangle \left\langle g_{j},y\right\rangle .
\]
\end{proof}
\begin{cor}
\label{cor:ke2}The system $\left\{ \left|e_{j}\left\rangle \right\langle e_{j}\right|\right\} _{j\in\mathbb{N}_{0}}$
is effective iff $\left\{ g_{j}\right\} _{j\in\mathbb{N}_{0}}$ is
a Parseval frame in $\mathscr{H}$.
\end{cor}

\begin{rem}
\label{rem:nf}We note that when $\mu$ is slice singular, then the
Fourier frequencies $\left\{ e_{n}\right\} _{n\in\mathbb{N}_{0}}$
is effective in $L^{2}\left(\mu\right)$, and every $f\in L^{2}\left(\mu\right)$
has Fourier series expansion.

This conclusion is based on (\ref{eq:C20}) and (\ref{eq:C21}) from
Corollaries \ref{cor:kg} \& \ref{cor:ke}. In more detail: Assume
$\mu$ is slice singular, and take $\mathscr{H}=L^{2}\left(\mu\right)$.
We may then think of \corref{ke} as a (generalized) Fourier expansion
result since every $f$ in the specified $L^{2}\left(\mu\right)$
space admits a non-orthogonal Fourier expansion in terms of explicit
coefficients and the standard Fourier functions $e_{n}$. Indeed,
the corresponding generalized Fourier coefficients are computed with
the use of the functions $g_{n}$ of the Kaczmarz algorithm, see eq.
(\ref{eq:C21}) and \corref{kg}.

We stress that while the coefficients in the expansion for $f$ are
explicitly given in (\ref{eq:C21}), this is nonetheless a non-orthogonal
expansion; see also \cite{2016arXiv160308852H,MR3796641}.
\end{rem}

\subsection{\label{subsec:rkac}Random Kaczmarz constructions and sequences of
projections}

In the discussion below, the word ``\emph{random}'' will refer to
a fixed probability space $\left(\Omega,\mathscr{F},\mathbb{P}\right)$,
where $\Omega$ is a set (sample space), $\mathscr{F}$ is a $\sigma$-algebra
(specified events), and $\mathbb{P}$ is a probability measure defined
on $\mathscr{F}$. \emph{Random variables} will then be measurable
functions on $\left(\Omega,\mathscr{F}\right)$. For example, if $\xi:\Omega\rightarrow\mathscr{B}\left(\mathscr{H}\right)$
is an operator valued random variable, measurability will then refer
to the $\sigma$-algebra of subsets in $\mathscr{B}\left(\mathscr{H}\right)$
which are w.r.t. the usual operator topology.

Equivalently, $\xi:\Omega\rightarrow\mathscr{B}\left(\mathscr{H}\right)$
is a random variable iff (Def.) for all pairs of vectors $x,y\in\mathscr{H}$,
then the functions 
\begin{equation}
\Omega\longrightarrow\mathbb{C},\quad\omega\longmapsto\left\langle x,\xi\left(\omega\right)y\right\rangle _{\mathscr{H}}
\end{equation}
are measurable w.r.t. the standard Borel $\sigma$-algebra $\mathscr{B}_{\mathbb{C}}$
of subsets of $\mathbb{C}$.

Given a probability space $\left(\Omega,\mathscr{F},\mathbb{P}\right)$
we shall denote the corresponding expectation $\mathbb{E}$, i.e.,
\begin{equation}
\mathbb{E}\left(\cdots\right)\overset{\text{Def.}}{=}\int_{\Omega}\left(\cdots\right)d\mathbb{P}.
\end{equation}
\\

\thmref{rkac} below is a stochastic variant of the classical Kaczmarz
algorithm; also see \thmref{kac}. For recent development and applications,
we refer to \cite{MR2113344,MR2721177,MR3210983,MR3450541,MR3345342,MR3424852,MR3439812,MR3847751,MR3796634}.

Let $\mathscr{H}$ be a Hilbert space. Given a family of selfadjoint
projections $\left\{ P_{j}\right\} _{j\in\mathbb{N}_{0}}$ in $\mathscr{H}$,
let $\xi:\Omega\rightarrow\mathscr{B}\left(\mathscr{H}\right)$ be
a random variable, such that 
\begin{equation}
\mathbb{P}\left(\xi=P_{j}\right)=p_{j},\;j\in\mathbb{N}_{0},\label{eq:rm1}
\end{equation}
where $p_{j}>0$, and $\sum_{j\in\mathbb{N}_{0}}p_{j}=1$.

Suppose further that there exists a constant $C$, $0<C<1$, such
that 
\begin{equation}
\mathbb{E}\left[\left\Vert \xi x\right\Vert ^{2}\right]:=\sum\nolimits _{j\in\mathbb{N}_{0}}p_{j}\left\Vert P_{j}x\right\Vert ^{2}\geq C\left\Vert x\right\Vert ^{2},\;\forall x\in\mathscr{H}.\label{eq:rm2}
\end{equation}

\begin{defn}
Let $\xi$, $\eta:\Omega\rightarrow\mathscr{B}\left(\mathscr{H}\right)$
be two operator-valued random variables. We say $\xi$ and $\eta$
are \emph{independent} iff (Def.) for all $x,y\in\mathscr{H}$, the
scalar valued random variables $\left\langle x,\xi y\right\rangle $,
and $\left\langle x,\eta y\right\rangle $ are independent.
\end{defn}

We shall use the standard abbreviation i.i.d. for independent, identically
distributed; also in the case of an indexed family of operator valued
random variables. In the present case, the common distribution is
specified by fixing the data in (\ref{eq:rm1}).

The key feature of our present \emph{randomized} Kaczmarz algorithm
is that it outputs a recursively generated sequence of operator valued
random variables; see (\ref{eq:rm3}) and (\ref{eq:rm4}). Each output,
in turn, will be a product of a specified i.i.d. system of projection
valued random variables. The latter i.i.d. system serves as input
into the algorithm.
\begin{thm}
\label{thm:rkac}Let $\left\{ \xi_{j}\right\} _{j\in\mathbb{N}_{0}}$
be an i.i.d. realization of $\xi$ from (\ref{eq:rm1}). Fix $\xi_{0}=P_{0}$,
and set 
\begin{align}
T_{n} & =\left(1-\xi_{n}\right)\left(1-\xi_{n-1}\right)\cdots\left(1-\xi_{0}\right),\;\text{and}\label{eq:rm3}\\
Q_{n} & =\xi_{n}\left(1-\xi_{n-1}\right)\cdots\left(1-\xi_{0}\right),\;Q_{0}=\xi_{0}.\label{eq:rm4}
\end{align}
Note that each product in (\ref{eq:rm3}) and (\ref{eq:rm4}) is an
operator-valued random variable.

Then, for all $x\in\mathscr{H}$, we have:
\begin{equation}
\lim_{n\rightarrow\infty}\mathbb{E}\text{\ensuremath{\left[\left\Vert T_{n}x\right\Vert ^{2}\right]}}=0.\label{eq:rm5}
\end{equation}
\end{thm}

\begin{proof}
For all $x\in\mathscr{H}$, we have
\[
T_{n}x=T_{n-1}x-\xi_{n}T_{n-1}x.
\]
But each $\xi_{n}$ is a random variable taking values in the set
of selfadjoint projections, as specified in (\ref{eq:rm1}), and so
\[
\left\Vert T_{n}x\right\Vert ^{2}=\left\Vert T_{n-1}x\right\Vert ^{2}-\left\Vert \xi_{n}T_{n-1}x\right\Vert ^{2}.
\]
It follows from (\ref{eq:rm2}) that
\begin{align*}
\mathbb{E}_{\xi_{1},\cdots,\xi_{n-1}}\left[\left\Vert T_{n}x\right\Vert ^{2}\right] & =\mathbb{E}_{\xi_{1},\cdots,\xi_{n-1}}\left[\left\Vert T_{n-1}x\right\Vert ^{2}\right]-\mathbb{E}_{\xi_{1},\cdots,\xi_{n-1}}\left[\left\Vert \xi_{n}T_{n-1}x\right\Vert ^{2}\right]\\
 & \leq\left\Vert T_{n-1}x\right\Vert ^{2}\left(1-C\right).
\end{align*}
Therefore, by taking expectation again, we get
\begin{eqnarray*}
\mathbb{E}\left[\left\Vert T_{n}x\right\Vert ^{2}\right] & \leq & \mathbb{E}\left[\left\Vert T_{n-1}x\right\Vert ^{2}\right]\left(1-C\right)\\
 & \leq & \mathbb{E}\left[\left\Vert T_{n-2}x\right\Vert ^{2}\right]\left(1-C\right)^{2}\\
 & \vdots\\
 & \leq & \mathbb{E}\left[\left\Vert T_{0}x\right\Vert ^{2}\right]\left(1-C\right)^{n}\\
 & = & \left\Vert x_{0}-x\right\Vert ^{2}\left(1-C\right)^{n}\rightarrow0,\;n\rightarrow\infty.
\end{eqnarray*}
\end{proof}
\begin{cor}
\label{cor:rkac}Let $T_{n}$ and $Q_{n}$ be as in (\ref{eq:rm3})--(\ref{eq:rm4}),
then the following hold.
\begin{enumerate}
\item For all $x\in\mathscr{H}$,
\begin{equation}
\lim_{n\rightarrow\infty}\mathbb{E}\left[\left\Vert x-\sum\nolimits _{j=0}^{n}Q_{j}x\right\Vert ^{2}\right]=0.\label{eq:rm6}
\end{equation}
\item For all $x,y\in\mathscr{H}$, 
\begin{equation}
\lim_{n\rightarrow\infty}\mathbb{E}\left[\left|\left\langle x,y\right\rangle -\sum\nolimits _{j=0}^{n}\left\langle x,Q_{j}^{*}Q_{j}y\right\rangle \right|^{2}\right]=0.\label{eq:rm7}
\end{equation}
\end{enumerate}
\end{cor}

\begin{proof}
The assertion (\ref{eq:rm6}) follows from (\ref{eq:rm5}) and (\ref{eq:cc7}).

By (\ref{eq:cc6}), we have $\left\Vert T_{n}x\right\Vert ^{2}=\left\langle x,T_{n}^{*}T_{n}x\right\rangle =\left\langle x,x\right\rangle -\sum_{j=0}^{n}\left\langle x,Q_{j}^{*}Q_{j}x\right\rangle $,
and so 
\[
\mathbb{E}\left[\left\langle x,T_{n}^{*}T_{n}x\right\rangle \right]\rightarrow0,\;n\rightarrow\infty.
\]
Now (\ref{eq:rm7}) follows from this and the polarization identity.
\end{proof}
\begin{rem}[Fusion frames, and measure frames]
 Our present equation (\ref{eq:rm2}) may be viewed as an instance
of what is now called \emph{fusion frames}, and developed extensively
by Casazza et al. \cite{MR2066823,MR2419707,MR2440135}, and by others.
In addition, we note that our present (\ref{eq:rm7}) is closely related
to a formulation a certain notion of \emph{measure frames}, see e.g.,
\cite{MR2147063,MR2964017,MR3526434,MR3688637}, and its extensions
in \cite{MR3800275}.
\end{rem}

\subsection{\label{subsec:axy}Solutions to $Ax=y$ in finite, and in infinite,
dimensional spaces}

A natural extension of the classical Kaczmarz algorithm is to solve
the equation 
\[
Ax=y,
\]
when $x,y$ are vectors in an infinite-dimensional Hilbert space $\mathscr{H}$,
and $A,A^{-1}$ are both bounded operators in $\mathscr{H}$; see
\figref{kap}.

Equivalently, when $\left\{ \varphi_{j}\right\} _{j\in\mathbb{N}}$
is an ONB (or a Parseval frame) in $\mathscr{H}$, we shall consider
the system of equations 
\begin{eqnarray*}
\left\langle \varphi_{j},Ax\right\rangle  & = & \left\langle \varphi_{j},y\right\rangle \\
 & \Updownarrow\\
\left\langle A^{*}\varphi_{j},x\right\rangle  & = & \left\langle \varphi_{j},y\right\rangle .
\end{eqnarray*}

\begin{question}
\label{que:Axy}Given the complex numbers $\left\langle A^{*}\varphi_{j},x\right\rangle $,
$j\in\mathbb{N}$, is it possible to recover $x$ using the Kaczmarz
method?
\end{question}

The closest analog to the finite-dimensional setting is the class
of Hilbert-Schmidt operators, and we shall recall the basics below.
\begin{defn}
Assume $\mathscr{H}$ is separable. $A:\mathscr{H}\rightarrow\mathscr{H}$
is \emph{Hilbert-Schmidt} iff (Def.) $\exists$ an ONB $\left\{ e_{i}\right\} _{i\in\mathbb{N}_{0}}$,
such that 
\begin{equation}
\sum\nolimits _{i}\left\Vert Ae_{i}\right\Vert ^{2}<\infty.
\end{equation}
We denote the set of all Hilbert-Schmidt operators in $\mathscr{H}$
by $HS\left(\mathscr{H}\right)$.
\end{defn}

Note that $A\in HS\left(\mathscr{H}\right)$ iff $A^{*}A$ is \emph{trace
class}, and for an ONB $\left\{ e_{i}\right\} $ , we have 
\begin{equation}
\sum\left\Vert Ae_{i}\right\Vert ^{2}=\sum\left\langle e_{i},A^{*}Ae_{i}\right\rangle =tr\left(A^{*}A\right).
\end{equation}

\begin{lem}
$HS\left(\mathscr{H}\right)$ $\simeq\mathscr{H}\otimes\overline{\mathscr{H}}$,
where $\overline{\mathscr{H}}$ denotes the conjugate Hilbert space.
\end{lem}

\begin{proof}
If $\left\{ e_{i}\right\} $ is an ONB, set $\left|e_{i}\left\rangle \right\langle e_{j}\right|$
w.r.t. the inner product 
\begin{equation}
\left(A,B\right)\longmapsto tr\left(A^{*}B\right),
\end{equation}
for all $A,B\in HS\left(\mathscr{H}\right)$. Hence, 
\begin{equation}
\left\langle A,B\right\rangle _{HS}=\sum_{i}\left\langle e_{i},A^{*}Be_{i}\right\rangle _{\mathscr{H}}\left(=tr\left(A^{*}B\right)\right).
\end{equation}

We shall show that 
\begin{equation}
HS\left(\mathscr{H}\right)\ominus\left\{ \left|e_{i}\left\rangle \right\langle e_{j}\right|\right\} _{\mathbb{N}_{0}\times\mathbb{N}_{0}}=0,\label{eq:hs4}
\end{equation}
i.e., $\left\{ A_{ij}:=\left|e_{i}\left\rangle \right\langle e_{j}\right|\right\} $
is \emph{total} in $HS\left(\mathscr{H}\right)$. To see this, note
that 
\[
tr\left(\left|u\left\rangle \right\langle v\right|\right)=\left\langle v,u\right\rangle _{\mathscr{H}},\;u,v\in\mathscr{H}.
\]
In fact, one checks that 
\begin{align*}
tr\left(\left|u\left\rangle \right\langle v\right|\right) & =\sum_{i}\left\langle e_{i},u\right\rangle \left\langle v,e_{i}\right\rangle \\
 & =\sum\left\langle v,u\right\rangle ,\;\text{by Parseval.}
\end{align*}
Now, if $B\in HS\left(\mathscr{H}\right)$, then 
\begin{align*}
\left\langle B,\left|e_{i}\left\rangle \right\langle e_{j}\right|\right\rangle _{HS} & =tr\left(B^{*}\left|e_{i}\left\rangle \right\langle e_{j}\right|\right)\\
 & =tr\left(\left|B^{*}e_{i}\left\rangle \right\langle e_{j}\right|\right)\\
 & =tr\left(\left|e_{i}\left\rangle \right\langle Be_{j}\right|\right)\\
 & =\left\langle Be_{j},e_{i}\right\rangle _{\mathscr{H}},\;\text{by \ensuremath{\left(\ref{eq:hs4}\right)}.}
\end{align*}
Therefore, if $\left\langle B,\left|e_{i}\left\rangle \right\langle e_{j}\right|\right\rangle _{HS}=0$,
for all $i,j\in\mathbb{N}_{0}$, then $B=0$; since 
\[
Be_{j}=\sum_{i}\left\langle e_{i},Be_{j}\right\rangle _{\mathscr{H}}e_{i}.
\]
\end{proof}
Now, back to \queref{Axy}. From earlier discussion, the answer depends
on whether the sequence $\left\{ A^{*}\varphi_{j}\right\} $ is effective.
In general, we do not get an effective sequence, even if $A$ is assumed
Hilbert-Schmidt. However, under certain conditions (see (\ref{eq:hp0}))
the \emph{random} Kaczmarz algorithm applies, and we get an approximate
sequence that converges to $x$ in expectation. See details below.
\begin{lem}
\label{lem:rax}Suppose $A$ is a bounded operator in $\mathscr{H}$
with bounded inverse. Fix a Parseval frame $\left\{ \varphi_{j}\right\} _{j\in\mathbb{N}}$
in $\mathscr{H}$, let $P_{j}$ be the projection onto $A^{*}\varphi_{j}$,
$j\in\mathbb{N}$.

Assume further that 
\begin{equation}
1/\left\Vert A^{-1}\right\Vert ^{2}<\sum\nolimits _{k}\left\Vert A^{*}\varphi_{k}\right\Vert ^{2}<\infty.\label{eq:hp0}
\end{equation}
Then, there exists a probability distribution $\left\{ p_{j}\right\} $
on $\left\{ P_{j}\right\} $, given by 
\begin{equation}
p_{j}=\left\Vert A^{*}\varphi_{j}\right\Vert ^{2}/\sum\nolimits _{k}\left\Vert A^{*}\varphi_{k}\right\Vert ^{2},\label{eq:hp}
\end{equation}
such that, for all $h\in\mathscr{H}$, 
\begin{equation}
\sum\nolimits _{j\in\mathbb{N}}p_{j}\left\Vert P_{j}h\right\Vert ^{2}\geq C\left\Vert h\right\Vert ^{2},\label{eq:hp1}
\end{equation}
where $C$ is a constant, $0<C<1$.
\end{lem}

\begin{proof}
For all $h\in\mathscr{H}$, we have: 
\begin{align*}
\left\Vert h\right\Vert ^{2} & =\left\Vert A^{-1}Ah\right\Vert ^{2}\\
 & \leq\left\Vert A^{-1}\right\Vert ^{2}\left\Vert Ah\right\Vert ^{2}\\
 & =\left\Vert A^{-1}\right\Vert ^{2}\sum\left|\left\langle \varphi_{j},Ah\right\rangle \right|^{2}=\left\Vert A^{-1}\right\Vert ^{2}\sum\left|\left\langle A^{*}\varphi_{j},h\right\rangle \right|^{2}\\
 & =\left\Vert A^{-1}\right\Vert ^{2}\sum\nolimits _{k}\left\Vert A^{*}\varphi_{k}\right\Vert ^{2}\sum\nolimits _{j}\underset{=p_{j}}{\underbrace{\frac{\left\Vert A^{*}\varphi_{j}\right\Vert ^{2}}{\sum_{k}\left\Vert A^{*}\varphi_{k}\right\Vert ^{2}}}}\left|\left\langle \frac{A^{*}\varphi_{j}}{\left\Vert A^{*}\varphi_{j}\right\Vert },h\right\rangle \right|^{2}\\
 & =\underset{=C^{-1}}{\underbrace{\left\Vert A^{-1}\right\Vert ^{2}\sum\nolimits _{k}\left\Vert A^{*}\varphi_{k}\right\Vert ^{2}}}\cdot\sum\nolimits _{j}p_{j}\left\Vert P_{j}h\right\Vert ^{2}.
\end{align*}
The desired conclusion follows from this.
\end{proof}
\begin{cor}
Let the setting be as in \lemref{rax}. An approximate solution to
$Ax=y$ is obtained recursively as follows:

Let $\xi:\Omega\rightarrow\mathscr{B}\left(\mathscr{H}\right)$ be
a random projection, s.t. $\mathbb{P}\left(\xi=P_{j}\right)=p_{j}$
(see (\ref{eq:hp})), and $\left\{ \xi_{j}\right\} $ be an i.i.d.
realization of $\xi$. Then, with $x_{0}\neq0$ fixed, and 
\begin{equation}
x_{j}:=x_{j-1}+\xi_{j}\left(x-x_{j-1}\right),\;j\in\mathbb{N},\label{eq:p1}
\end{equation}
we have:
\begin{equation}
\lim_{j\rightarrow\infty}\mathbb{E}\left[\left\Vert x_{j}-x\right\Vert ^{2}\right]=0.\label{eq:p2}
\end{equation}
Note that, in (\ref{eq:p1}) if $\xi=P_{k}$, then 
\[
\xi x=\frac{\left\langle A^{*}\varphi_{k},x\right\rangle }{\left\Vert A^{*}\varphi_{k}\right\Vert ^{2}}A^{*}\varphi_{k}=\frac{\left\langle \varphi_{k},y\right\rangle }{\left\Vert A^{*}\varphi_{k}\right\Vert ^{2}}A^{*}\varphi_{k}.
\]
\end{cor}

\begin{proof}
By \lemref{rax}, the estimate (\ref{eq:hp1}) holds with the probabilities
specified in (\ref{eq:hp}). See also condition (\ref{eq:rm2}). Moreover,
it follows from (\ref{eq:p1}) that 
\[
x-x_{j}=\left(1-\xi_{j}\right)\left(1-\xi_{j-1}\right)\cdots\left(1-\xi_{1}\right)x_{0}.
\]
Therefore, by \thmref{rkac}, the assertion in (\ref{eq:p2}) holds
(with a suitable choice of index $j$).
\end{proof}

\section{System of isometries}

Below we discuss a particular aspect of our problem where the polydisk
$\mathbb{D}^{d}$ will play an important role. As outlined below,
the polydisk is a natural part of our harmonic analysis of frame-approximation
questions in the Hilbert space $L^{2}(\mu)$, where $\mu$ is in a
suitable class of IFS-measures, i.e., the multivariable setting for
fractal measures.
\begin{lem}
\label{lem:viso}Fix $d>1$, and let $\mathbb{D}^{d}$ be the polydisk.
Let $H_{2}\left(\mathbb{D}^{d}\right)$ be the corresponding Hardy
space. Let $\mu$ be a Borel probability measure on $\mathbb{T}^{d}\simeq\left[0,1\right]^{d}$.
Then there is a bijective correspondence between:
\begin{enumerate}
\item \label{enu:pfs1}isometries $V:L^{2}\left(\mu\right)\rightarrow H_{2}\left(\mathbb{D}^{d}\right)$;
and
\item \label{enu:pfs2}Parseval frames $\left\{ g_{n}\right\} $ in $L^{2}\left(\mu\right)$.
\end{enumerate}
The correspondence is as follows:
\end{lem}

(\ref{enu:pfs1})$\rightarrow$(\ref{enu:pfs2}). Given $V$, isometric;
set $g_{n}:=V^{*}\left(z^{n}\right)$, where $n\in\mathbb{N}_{0}^{d}$.

(\ref{enu:pfs2})$\rightarrow$(\ref{enu:pfs1}). Given $\left\{ g_{n}\right\} $
a fixed Parseval frame in $L^{2}\left(\mu\right)$, set 
\[
\left(Vf\right)\left(z\right)=\sum_{n\in\mathbb{N}_{0}^{d}}\left\langle g_{n},f\right\rangle _{L^{2}\left(\mu\right)}z^{n},\;z\in\mathbb{D}^{d}.
\]

\begin{proof}
The fact that there is a correspondence between isometries and Parseval
frames is general. Let $\mathscr{H}_{1}$ be a separable Hilbert space,
then there is a bijective correspondence between the following two:
\begin{enumerate}
\item \label{enu:pf1}A Parseval frame $\left(g_{n}\right)_{n\in\mathbb{N}}$
in $\mathscr{H}_{1}$ (with a suitable choice of index);
\item \label{enu:pf2}A pair $\left(\mathscr{H}_{2},V\right)$, where $\mathscr{H}_{2}$
is a Hilbert space, and $V:\mathscr{H}_{1}\rightarrow\mathscr{H}_{2}$
is isometric.
\end{enumerate}
(Note that there is a similar result for Bessel frames as well.) The
correspondence is as follows.

Given a Parseval frame $\left(g_{n}\right)_{n\in\mathbb{N}}$ in $\mathscr{H}_{1}$,
take $\mathscr{H}_{2}:=l^{2}\left(\mathbb{N}\right)$, and set $Vf=\sum_{n}\left\langle g_{n},f\right\rangle _{\mathscr{H}_{1}}\delta_{n}$,
where $\left\{ \delta_{n}\right\} _{n\in\mathbb{N}}$ is the standard
ONB in $l^{2}\left(\mathbb{N}\right)$.

Conversely, let $\left(\mathscr{H}_{2},V\right)$ be such that $\mathscr{H}_{1}\xrightarrow{\;V\;}\mathscr{H}_{2}$
is isometric. Choose an ONB $\left\{ \beta_{n}\right\} _{n\in\mathbb{N}}$
in $\mathscr{H}_{2}$, and set $g_{n}=V^{*}\beta_{n}$. Then $\left\{ g_{n}\right\} $
is a Parseval frame in $\mathscr{H}_{1}$. Indeed, for all $h\in\mathscr{H}_{1}$,
one checks that, 
\begin{align*}
\sum_{n}\left|\left\langle g_{n},h\right\rangle _{\mathscr{H}_{1}}\right|^{2} & =\sum_{n}\left|\left\langle V^{*}\beta_{n},h\right\rangle _{\mathscr{H}_{1}}\right|^{2}\\
 & =\sum_{n}\left|\left\langle \beta_{n},Vh\right\rangle _{\mathscr{H}_{2}}\right|^{2}\\
 & =\left\Vert Vh\right\Vert _{\mathscr{H}_{2}}^{2}=\left\Vert h\right\Vert _{\mathscr{H}_{1}}^{2}.
\end{align*}

The lemma follows by setting $\mathscr{H}_{1}=L^{2}\left(\mu\right)$,
and $\mathscr{H}_{2}=H_{2}\left(\mathbb{D}^{d}\right)$.
\end{proof}
\begin{defn}
Fix $d>1$. For all $x\in\mathbb{T}^{d}$, and all $z\in\mathbb{D}^{d}$,
let 
\begin{equation}
K^{\ast}\left(z,x\right)=\prod_{j=1}^{d}\frac{1}{1-z_{j}\overline{e\left(x_{j}\right)}}.
\end{equation}
Let $\mu\in\mathcal{M}\left(\mathbb{T}^{d}\right)$, and set
\begin{align}
\left(C_{\mu}f\right)\left(z\right) & =\int_{\mathbb{T}^{d}}f\left(x\right)K^{\ast}\left(z,x\right)d\mu\left(x\right)\label{eq:nc1}\\
 & =\sum_{n\in\mathbb{N}^{d}}\widehat{fd\mu}\left(n\right)z^{n}.\nonumber 
\end{align}
In particular, 
\begin{equation}
\left(C_{\mu}1\right)\left(z\right)=\sum_{n\in\mathbb{N}_{0}^{d}}\widehat{\mu}\left(n\right)z^{n},
\end{equation}
where $\widehat{\mu}\left(n\right)=\int_{\mathbb{T}^{d}}\overline{e_{n}\left(x\right)}d\mu\left(x\right)$,
$n\in\mathbb{N}_{0}^{d}$.
\end{defn}

Let $L^{2}\left(\mu\right)\left(=L^{2}\left(\mathbb{T}^{2},\mu\right)\right)$
be as above, where $\mu\in\mathcal{M}^{+}\left(\mathbb{T}^{2}\right)$,
$\xi=\mu\circ\pi_{1}^{-1}$, and $\mu$ assumes a disintegration $d\mu=\int\sigma^{x}\left(dy\right)d\xi\left(x\right)$.
\begin{thm}[see e.g., \cite{MR1289670,MR3411042}]
\label{thm:ssiso} Assume $\mu$ is slice singular. There are then
two associated isometries:
\begin{equation}
L^{2}\left(\xi\right)\xrightarrow{\;V_{\xi}\;}H_{2}\left(\mathbb{D}\right),\quad\left(V_{\xi}f\right)\left(z\right)=\frac{\left(C_{\xi}f\right)\left(z\right)}{\left(C_{\xi}1\right)\left(z\right)},\label{eq:vxi}
\end{equation}
and 
\begin{equation}
L^{2}\left(\sigma^{x}\right)\xrightarrow{\;V_{\sigma^{x}}\;}H_{2}\left(\mathbb{D}\right),\quad\left(V_{\sigma^{x}}f\right)\left(z\right)=\frac{\left(C_{\sigma^{x}}f\right)\left(z\right)}{\left(C_{\sigma^{x}}1\right)\left(z\right)}.\label{eq:vsig}
\end{equation}
\end{thm}

\begin{proof}[Proof sketch]
Let $\nu$ be a positive Borel measure on $\left[0,1\right]$, and
$C_{\mu}$ be the Cauchy transform from (\ref{eq:nc1}). Assume $\nu$
is singular.

Then, by F.M Riesz (see \thmref{FM}), the set $\left\{ e_{n}\right\} _{n\in\mathbb{N}_{0}}$
is total in $L^{2}\left(\nu\right)$. Moreover, it follows from \cite{MR2263965},
that $\left\{ e_{n}\right\} _{n\in\mathbb{N}_{0}}$ is effective.
Thus, every $f\in L^{2}\left(\nu\right)$ has (non-orthogonal) Fourier
expansion 
\[
f=\sum_{n\in\mathbb{N}_{0}}\left\langle g_{n},f\right\rangle _{L^{2}\left(\nu\right)}e_{n},
\]
where $\left\{ g_{n}\right\} $ is the Parseval frame in $L^{2}\left(\nu\right)$
constructed from Kaczmarz' algorithm. See also \remref{nf}. One may
verify that 
\[
\left(V_{\nu}f\right)\left(z\right)=\frac{C_{\nu}f}{C_{\nu}1}=\sum_{n\in\mathbb{N}_{0}}\left\langle g_{n},f\right\rangle _{L^{2}\left(\nu\right)}z^{n},
\]
and so $V_{\nu}:L^{2}\left(\nu\right)\rightarrow H_{2}\left(\mathbb{D}\right)$
is isometric.

The theorem follows from this, and the assumption that $\mu$ is slice
singular.
\end{proof}
\begin{cor}
The mapping 
\[
V_{\mu}:L^{2}\left(\mu\right)\longrightarrow H_{2}\left(\mathbb{D}^{2}\right)\left(=H_{2}\left(\mathbb{D}\right)\otimes H_{2}\left(\mathbb{D}\right)\right)
\]
given by 
\begin{equation}
\left(V_{\mu}F\right)\left(z_{1},z_{2}\right)=V_{\xi}\left(\left(V_{\sigma^{x}\left(\cdot\right)}F\left(x,\cdot\right)\right)\left(z_{2}\right)\right)\left(z_{1}\right)
\end{equation}
is isometric. It follows that $\left\{ g_{n}:=V_{\mu}^{*}\left(z^{n}\right)\right\} _{n\in\mathbb{N}_{0}^{2}}$
is a Parseval frame in $L^{2}\left(\mu\right)$.
\end{cor}

\begin{proof}
Follows from \thmref{ssiso} and \lemref{viso}.
\end{proof}
\begin{rem}
From the above discussion, we see that if $V:L^{2}\left(\mu\right)\rightarrow H_{2}\left(\mathbb{D}^{2}\right)$
is an isometry, then $\left\{ g_{n}:=V^{*}\left(z^{n}\right)\right\} _{n\in\mathbb{N}_{0}^{2}}$
is a Parseval frame in $L^{2}\left(\mu\right)$. This implication
holds in general. Since there are ``many'' such isometries, it follows
that there are ``many'' Parseval frames. For more details, see \cite{2016arXiv160308852H,MR3796641,HERR2018}
and the reference therein.
\end{rem}

\begin{lem}
Let $K$ be a kernel on $\mathbb{D}^{d}$, and $\mu$ be a measure
on $\mathbb{T}^{d}$. Then for all $z\in\mathbb{D}^{d}$, we have
$\lim_{r\rightarrow1}K\left(z,re\left(x\right)\right)=K^{\ast}\left(z,x\right)$,
a.a. $x\in\mathbb{T}^{d}$; and 
\[
V_{\mu}^{*}\left(K\left(\cdot,z\right)\right)=K^{\ast}\left(z,x\right),
\]
a.a. $x$ w.r.t. $\mu$.
\end{lem}

\section{\label{sec:gifs}General iterated function system (IFS)-theory}

In this section we turn to an analysis of the IFS measures (see e.g.,
\cite{MR625600,MR1656855,MR2319756,MR2431670,2016arXiv160308852H,MR3882025}),
as introduced in Sections \ref{sec:Intro} and \ref{sec:SSM} (see
(\ref{eq:E3}) below). The notion of iterated function systems (IFS)
for the case of measures fits the following general idea of patterns
with self-similarity across different scales. Also here, the IFS-measures
are created by recursive repetition of a simple process in an ongoing
feedback loop.

Recall that an IFS measure is obtained from a recursive algorithm
involving successive iteration of a finite system of maps in a metric
space. IFS systems are self-similar because the same fixed choice
of \textquotedblleft scaling\textquotedblright{} mappings is used
in each step of the algorithm. (The simplest IFS measures arise from
the standard Cantor construction applied to a finite interval. But
the idea works much more generally.) Then the chosen finite index-set
for the mappings is called an alphabet, denoted $B$. We shall analyze
here the IFS measures with the aid of symbolic dynamics on a probability
space $\Omega$, made up of infinite words in $B$. Then a fixed choice
of probability weights on $B$ leads to an associated infinite product
measure, called $\mathbb{P}$, on $\Omega$, see (\ref{eq:d2}). By
Kakutani's theorem, distinct weights yield mutually singular infinite
product measures.

We shall construct a random variable $X$ on $\Omega$ such that the
IFS then arises as the image under $X$, and the IFS measure $\mu$
becomes the distribution of $X$. Intuitively, $X$ is an infinite
address map; see also eq (\ref{eq:d2}) and \thmref{pp}. While the
choice of such system of maps could be rather general, we shall restrict
attention here to the case of a finite number of contractive affine
mappings in $\mathbb{R}^{d}$, $d$ fixed; see e.g., (\ref{eq:F2})
for the case of the standard Sierpinski gasket, where $d=2$. In this
case, the associated maximal entropy measure $\mu$ (see (\ref{eq:si3}))
is a probability measure prescribed by the uniform distribution on
$B$.\\

Let $\left(M,d\right)$ be a complete metric space. Fix an alphabet
$B=\left\{ b_{1},\cdots,b_{N}\right\} $, $N\geq2$, and let $\left\{ \tau_{b}\right\} _{b\in B}$
be a contractive IFS with attractor $W\subset M$, i.e., 
\begin{equation}
W=\bigcup_{b}\tau_{b}\left(W\right).\label{eq:e1}
\end{equation}
In fact, $W$ is uniquely determined by (\ref{eq:e1}).

Let $\left\{ p_{b}\right\} _{b\in B}$, $p_{b}>0$, $\sum_{b\in B}p_{b}=1$,
be fixed. Set $\Omega=B^{\mathbb{N}}$, equipped with the product
topology. Let 
\begin{equation}
\mathbb{P}=\vartimes{}_{1}^{\infty}p=\underset{\aleph_{0}\text{ product measure}}{\underbrace{p\times p\times p\cdots\cdots}}\label{eq:d2}
\end{equation}
be the infinite-product measure on $\Omega$ (see \cite{MR0014404,MR562914}).

In this section, we construct a random variable $X:\Omega\rightarrow M$
with value in $M$ (a measure space $\left(M,\mathscr{B}_{M}\right)$),
such that the distribution $\mu:=\mathbb{P}\circ X^{-1}$ is a Borel
probability measure supported on $W$, satisfying 
\begin{equation}
\mu=\sum_{b\in B}p_{b}\,\mu\circ\tau_{b}^{-1}.\label{eq:E3}
\end{equation}
That is, $\mu$ is the IFS measure.
\begin{thm}
\label{thm:pp}For points $\omega=\left(b_{i_{1}},b_{i_{2}},b_{i_{3}},\cdots\right)\in\Omega$
and $k\in\mathbb{N}$, set 
\begin{align}
\omega\big|_{k} & =\left(b_{i_{1}},b_{i_{2}},\cdots,b_{i_{k}}\right),\;\text{and}\label{eq:d3}\\
\tau_{\omega|_{k}} & =\tau_{b_{i_{k}}}\circ\cdots\circ\tau_{b_{i_{2}}}\circ\tau_{b_{i_{1}}}.\label{eq:d4}
\end{align}
Then $\bigcap_{k=1}^{\infty}\tau_{\omega|_{k}}$$\left(M\right)$
is a singleton, say $\left\{ x\left(\omega\right)\right\} $. Set
$X\left(\omega\right)=x\left(\omega\right)$, i.e., 
\begin{equation}
\left\{ X\left(\omega\right)\right\} =\bigcap_{k=1}^{\infty}\tau_{\omega|_{k}}\left(M\right);\label{eq:d5}
\end{equation}
then:
\begin{enumerate}
\item $X:\Omega\rightarrow M$ is an $\left(M,d\right)$-valued random variable.
\item \label{enu:pp2}The distribution of $X$, i.e., the measure 
\begin{equation}
\mu=\mathbb{P}\circ X^{-1}\label{eq:d6}
\end{equation}
is the unique Borel probability measure on $\left(M,d\right)$ satisfying:
\begin{equation}
\mu=\sum_{b\in B}p_{b}\,\mu\circ\tau_{b}^{-1};\label{eq:d7}
\end{equation}
equivalently, 
\begin{equation}
\int_{M}fd\mu=\sum_{b\in B}p_{b}\int_{M}\left(f\circ\tau_{b}\right)d\mu,\label{eq:d8}
\end{equation}
holds for all Borel functions $f$ on $M$.
\item The support $W_{\mu}=supp\left(\mu\right)$ is the minimal closed
set (IFS), $\neq\emptyset$, satisfying 
\begin{equation}
W_{\mu}=\bigcup_{b\in B}\tau_{b}\left(W_{\mu}\right).\label{eq:d9}
\end{equation}
\end{enumerate}
\end{thm}

\begin{proof}
We shall make use of standard facts from the theory of iterated function
systems (IFS), and their measures; see e.g., \cite{MR625600,MR1656855,MR2431670}.

\emph{Monotonicity}: When $\omega\in\Omega$ is fixed, then $\tau_{\omega|_{k}}\left(M\right)$
is a monotone family of compact subsets in $M$ s.t. 
\begin{equation}
\tau_{\omega|_{k+1}}\left(M\right)\subset\tau_{\omega|_{k}}\left(M\right).\label{eq:d10}
\end{equation}
Since $\tau_{b}$ is strictly contractive for all $b\in B$, we get
\begin{equation}
\lim_{k\rightarrow\infty}diameter\left(\tau_{\omega|_{k}}\left(M\right)\right)=0,\label{eq:d11}
\end{equation}
and so the intersection in (\ref{eq:d5}) is a singleton depending
only on $\omega$.

\emph{The $\sigma$-algebras on $\left(\Omega,\mathbb{P}\right)$
and $\left(X,d\right)$}: The $\sigma$-algebra of subsets of $\Omega$
is generated by cylinder sets. Specifically, if $f=\left(b_{i_{1}},b_{i_{2}},\cdots,b_{i_{k}}\right)$
is a \emph{finite word}, the corresponding cylinder set is 
\begin{equation}
E\left(f\right)=\left\{ \omega\in\Omega\mid\omega_{j}=b_{i_{j}},\;1\leq j\leq k\right\} \subset\Omega.\label{eq:d12}
\end{equation}
The Borel $\sigma$-algebra on $M$ is determined from the fixed metric
$d$ on $M$.

The measure $\mathbb{P}\left(=\mathbb{P}_{p}\right)$ is specified
by its values on cylinder sets; i.e, set 
\begin{equation}
\mathbb{P}\left(E\left(f\right)\right)=p_{b_{i_{1}}}p_{b_{i_{2}}}\cdots p_{b_{i_{k}}}=:p_{f}.\label{eq:d13}
\end{equation}
Also see e.g., \cite{MR735967}.

\emph{Proof of (\ref{eq:d7})}. The argument is based on the following:
On $\Omega$, introduce the \emph{shifts} $\widetilde{\tau}_{b}\left(b_{i_{1}},b_{i_{2}},b_{i_{3}},\cdots\right)=\left(b,b_{i_{1}},b_{i_{2}},b_{i_{3}},\cdots\right)$,
$b\in B$. Let $X$ be as in (\ref{eq:d5})-(\ref{eq:d6}), then 
\begin{equation}
\tau_{b}\circ X=X\circ\widetilde{\tau}_{b},\label{eq:d14}
\end{equation}
which is immediate from (\ref{eq:d5}).

\[
\xymatrix{\Omega\ar[rr]^{X}\ar[d]_{\widetilde{\tau}_{b}} &  & M\ar[d]^{\tau_{b}}\\
\Omega\ar[rr]_{X} &  & M
}
\]

We now show (\ref{eq:d8}), equivalently (\ref{eq:d7}). Let $f$
be a Borel function on $M$, then 
\begin{alignat*}{2}
\int_{M}f\,d\mu & =\int_{\Omega}\left(f\circ X\right)d\mathbb{P} & \quad & \text{\ensuremath{\left(\text{by }\left(\ref{eq:d6}\right)\right)}}\\
 & =\sum_{b\in B}p_{b}\int_{\Omega}f\circ X\circ\widetilde{\tau}_{b}\:d\mathbb{P} &  & \left(\begin{matrix}\text{since \ensuremath{\mathbb{P}} is the product}\\
\text{ measure \ensuremath{\vartimes_{1}^{\infty}p}, see \ensuremath{\left(\ref{eq:d13}\right)}}
\end{matrix}\right)\\
 & =\sum_{b\in B}p_{b}\int_{\Omega}f\circ\tau_{b}\circ X\:d\mathbb{P} &  & \text{\ensuremath{\left(\text{by }\left(\ref{eq:d14}\right)\right)}}\\
 & =\sum_{b\in B}p_{b}\int_{M}f\circ\tau_{b}\:d\mu &  & \text{\ensuremath{\left(\text{by }\left(\ref{eq:d6}\right)\right)}}
\end{alignat*}
which is the desired conclusion.
\end{proof}
In general, the random variable $X:\Omega\rightarrow W$ (see (\ref{eq:d5}))
is not 1-1, but it is always onto. It is 1-1 when the IFS is non-overlap;
see \defref{no} below.
\begin{defn}
\label{def:no}We say that $\left(\tau_{b},W\right)$ is ``non-overlap''
iff for all $b,b'\in B$, with $b\neq b'$, we have $\tau_{b}\left(W\right)\cap\tau_{b'}\left(W\right)=\emptyset$.
\end{defn}

\begin{cor}
\label{cor:kuk}Assume $p\neq p'$, i.e., $p_{b}\neq p_{b}'$, for
some $b\in B$. (Recall that $\sum_{b\in B}p_{b}=\sum_{b\in B}p_{b}'=1$,
$p_{b}$, $p_{b}'>0$.) Let $\mathbb{P}=\vartimes_{1}^{\infty}p$,
and $\mathbb{P}'=\vartimes_{1}^{\infty}p'$ be the corresponding infinite
product measures; and let $\mu=\mathbb{P}\circ X^{-1}$, $\mu'=\mathbb{P}'\circ X^{-1}$
be the respective distributions. Then $\mu$ and $\mu'$ are mutually
singular.
\end{cor}

\begin{proof}
This is an application of Kakutani's theorem on infinite product measures.
See \cite{MR0014404,MR0023331}.
\end{proof}
\begin{rem}[Affine IFSs]
\begin{flushleft}
Let $B=\left\{ b_{1},\cdots,b_{N}\right\} $ be a subset of $\mathbb{R}^{d}$,
and fix a $d\times d$ matrix $M$. Assume $M$ is expansive, i.e.,
$\left|\lambda\right|>1$, for all eigenvalues $\lambda$ of $M$.
Then the mapping $\Omega=\left\{ 1,\cdots,N\right\} ^{\mathbb{N}}\xrightarrow{\;X\;}W_{B}$
from (\ref{eq:d5}) is given by 
\[
\omega=\left(i_{1},i_{2},i_{3}\cdots\right)\longmapsto x:=\sum_{j=1}^{\infty}M^{-j}b_{i_{j}}.
\]
Note that $x$ has a random expansion, with the alphabets $b_{i}\in B$,
as a sequence of i.i.d. random variables with distribution $p=\left(p_{1},\cdots,p_{N}\right)$.
\par\end{flushleft}
\end{rem}

\section{\label{sec:Sie}Sierpinski and random power series}

Given a probability measure $\mu$ on $I^{d}$ where $I=\left[0,1\right]$,
a key property that $\mu$ may, or may not, have is that the Fourier
frequencies $\left\{ e_{n}\right\} _{n\in\mathbb{N}^{d}}$ are \emph{total}
in $L^{2}\left(\mu\right)$, i.e., that the closed span of $\left\{ e_{n}\right\} _{n\in\mathbb{N}^{d}}$
is $L^{2}\left(\mu\right)$.

The result in $d=1$, that, if $\nu$ on $I$ is singular, then the
set $\left\{ e_{n}\right\} _{n\in\mathbb{N}_{0}}$ is total in $L^{2}\left(\nu\right)$,
fails for $d=2$. There are examples when $\mu$ on $I^{2}$ is positive,
singular w.r.t. the 2D Lebesgue measure, but $\left\{ e_{n}\right\} _{n\in\mathbb{N}_{0}^{2}}$
is \emph{not} total in $L^{2}\left(\mu\right)$.
\begin{example}
Take $\mu=\lambda_{1}\times\nu$ (see \figref{pm}), where $\lambda_{1}$
is Lebesgue measure and $\nu$ is a singular measure in $I$, then
$\left\{ e_{n}\right\} _{n\in\mathbb{N}_{0}^{2}}$ is \emph{not} total
in $L^{2}\left(\mu\right)$.
\end{example}

\begin{figure}[H]
\begin{centering}
\includegraphics[width=0.2\columnwidth]{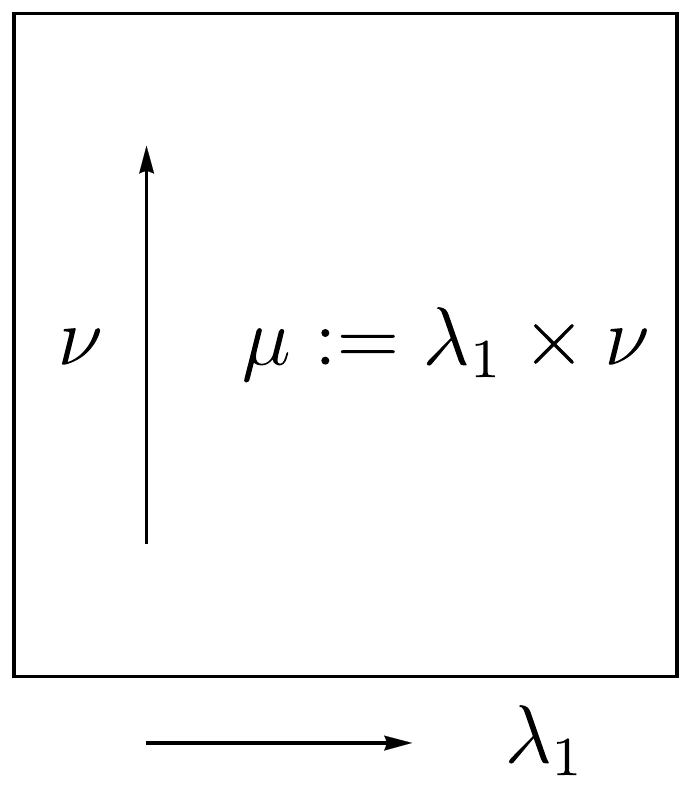}
\par\end{centering}
\caption{\label{fig:pm}$\lambda_{1}=$ Lebesgue, $\nu\perp\lambda_{1}$}
\end{figure}

For the Sierpinski case (affine IFS), with the Sierpinski measure
$\mu$, total does hold in $L^{2}\left(\mu\right)$. See details below.

Let the alphabets be 
\begin{equation}
B=\left\{ b_{0},b_{1},b_{2}\right\} :=\left\{ \begin{bmatrix}0\\
0
\end{bmatrix},\begin{bmatrix}1\\
0
\end{bmatrix},\begin{bmatrix}0\\
1
\end{bmatrix}\right\} .\label{eq:F1}
\end{equation}
Set
\begin{equation}
M=\begin{bmatrix}2 & 0\\
0 & 2
\end{bmatrix},\;\text{and}\quad\tau_{j}\left(x\right)=M^{-1}\left(x+b_{j}\right).\label{eq:F2}
\end{equation}
The Sierpinski gasket (\figref{isp}) is the IFS attractor $W$ satisfying
\[
W=\bigcup_{j=0}^{2}\tau_{j}\left(W\right).
\]

\begin{figure}
\begin{tabular}{ccc}
\includegraphics[width=0.2\columnwidth]{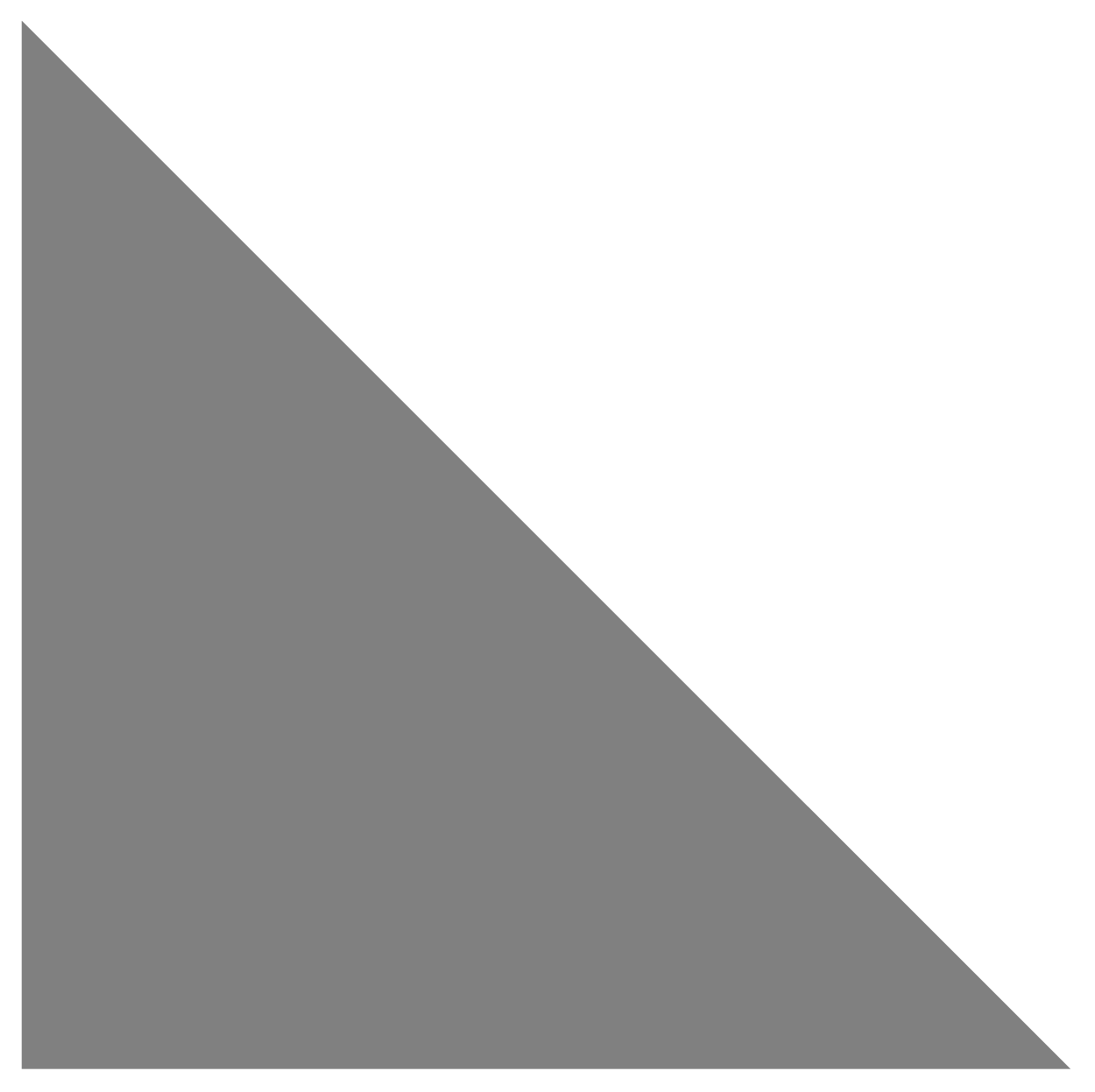} & \includegraphics[width=0.2\columnwidth]{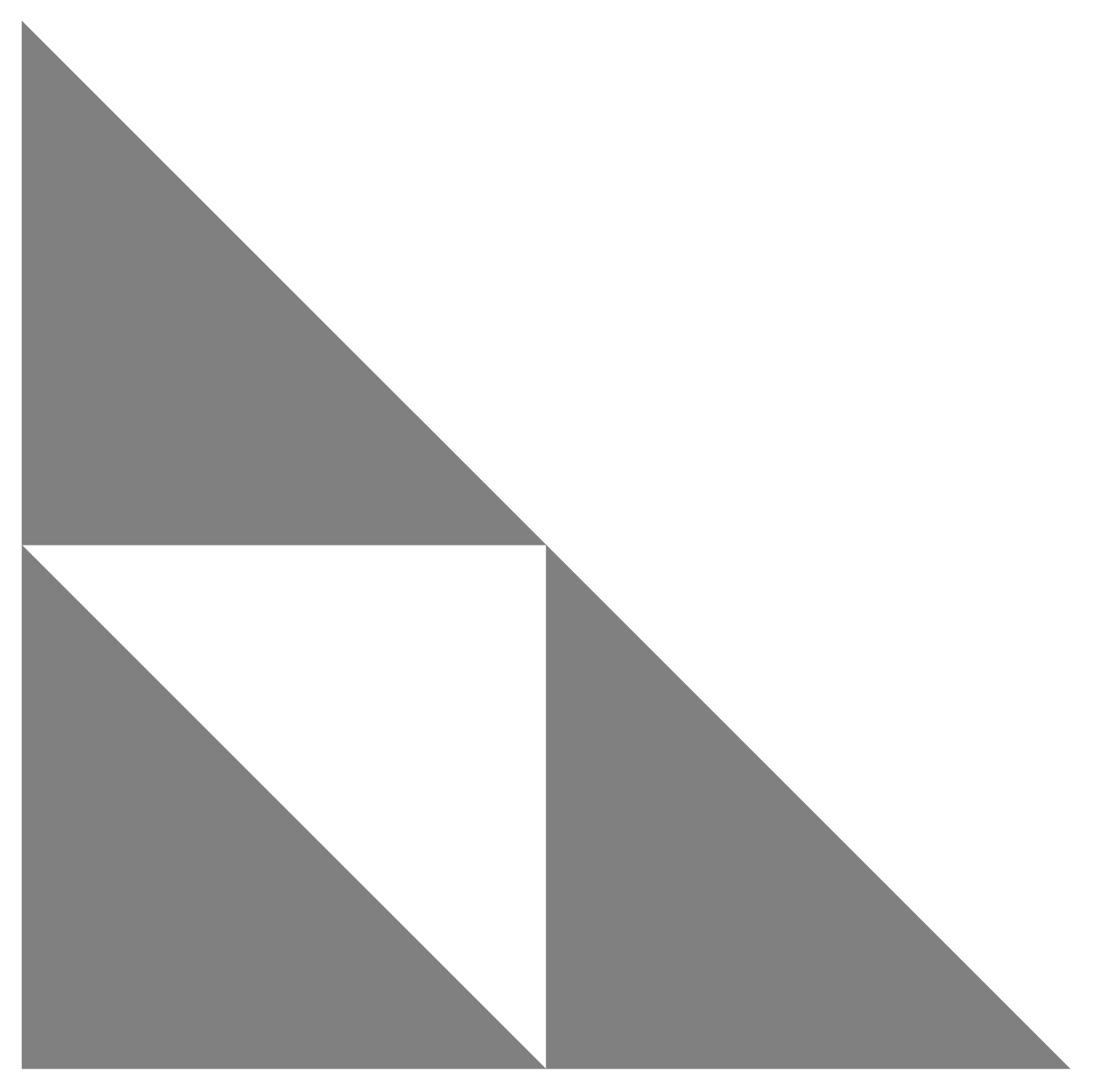} & \includegraphics[width=0.2\columnwidth]{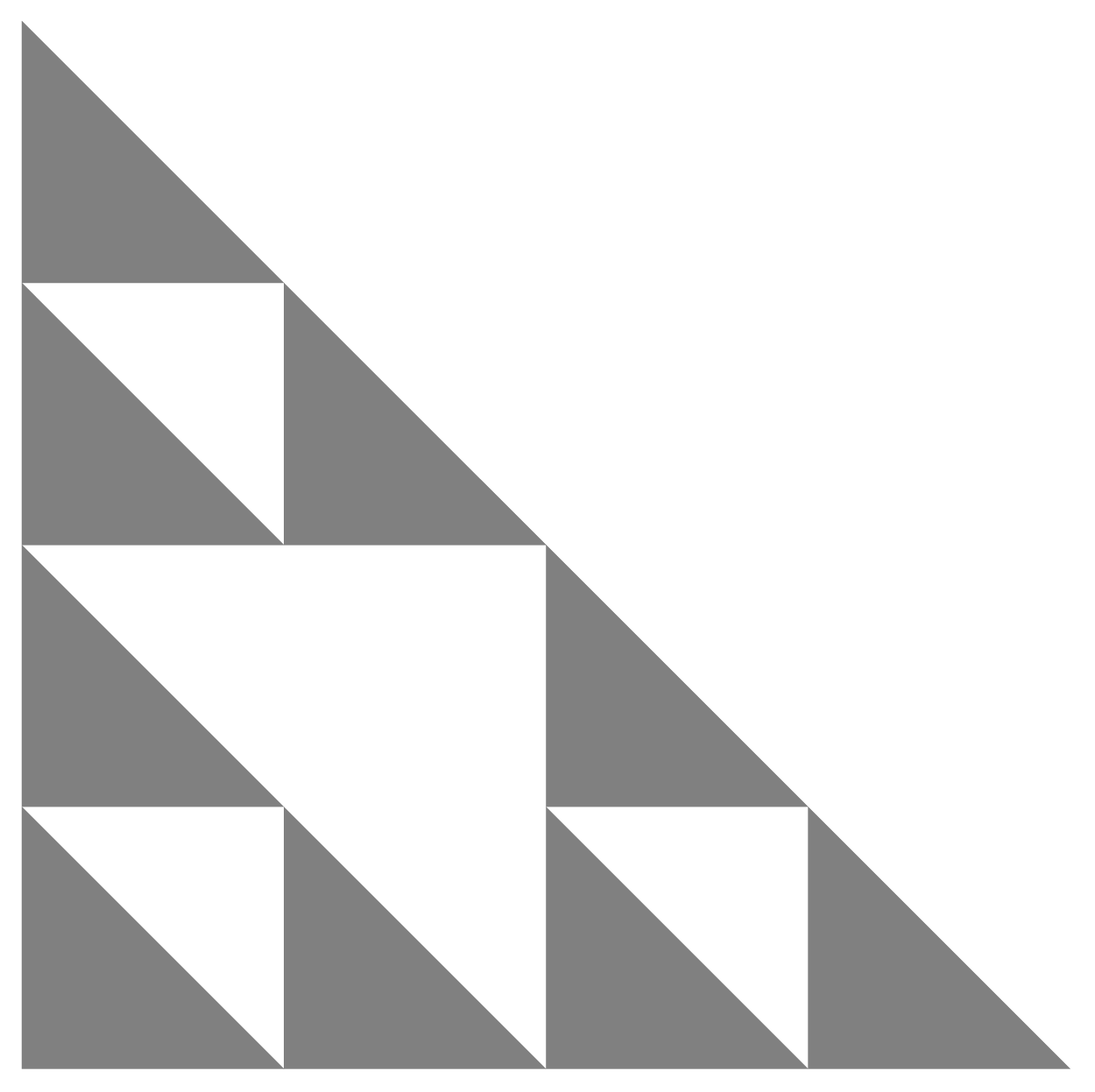}\tabularnewline
\includegraphics[width=0.2\columnwidth]{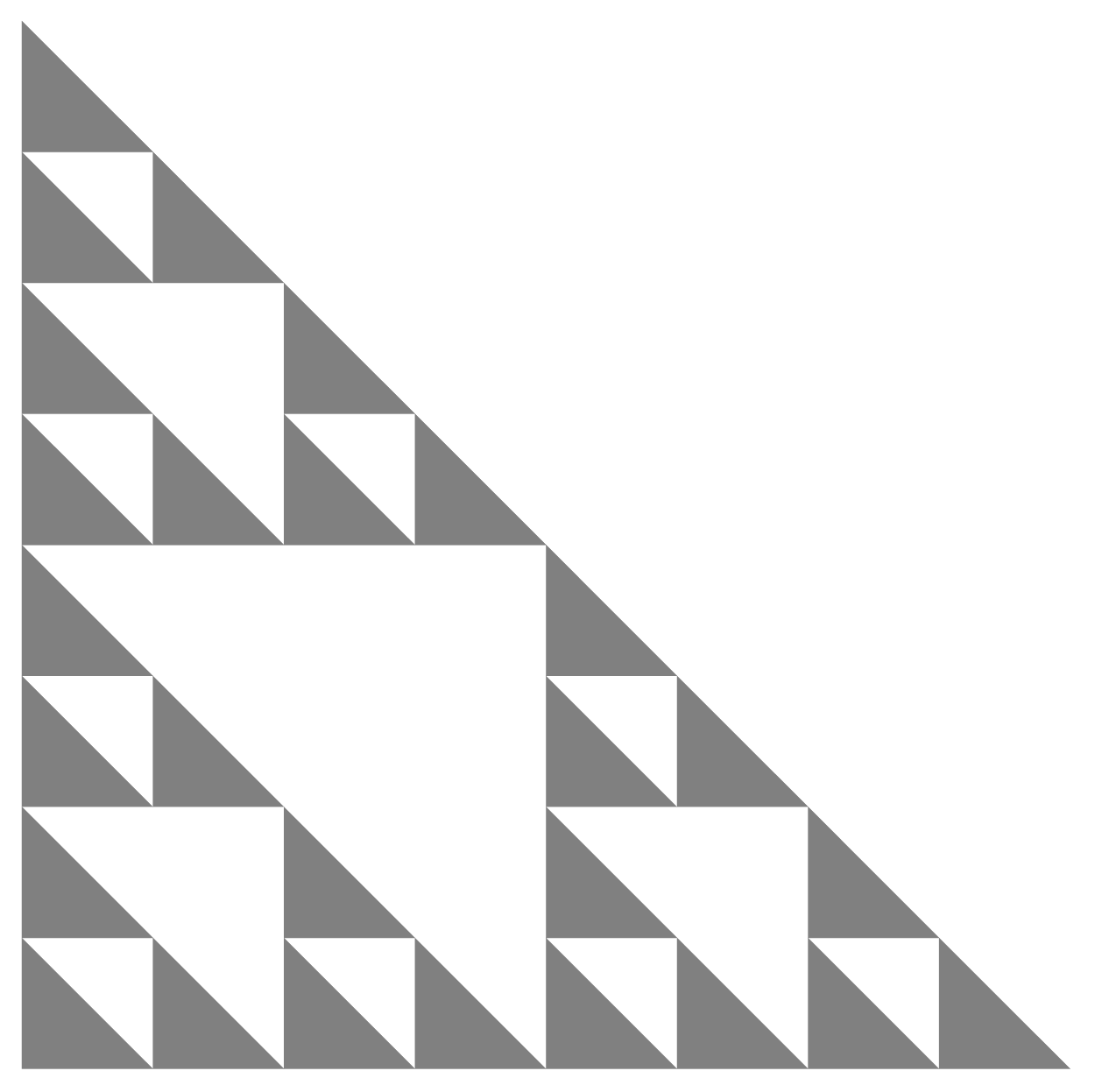} & \includegraphics[width=0.2\columnwidth]{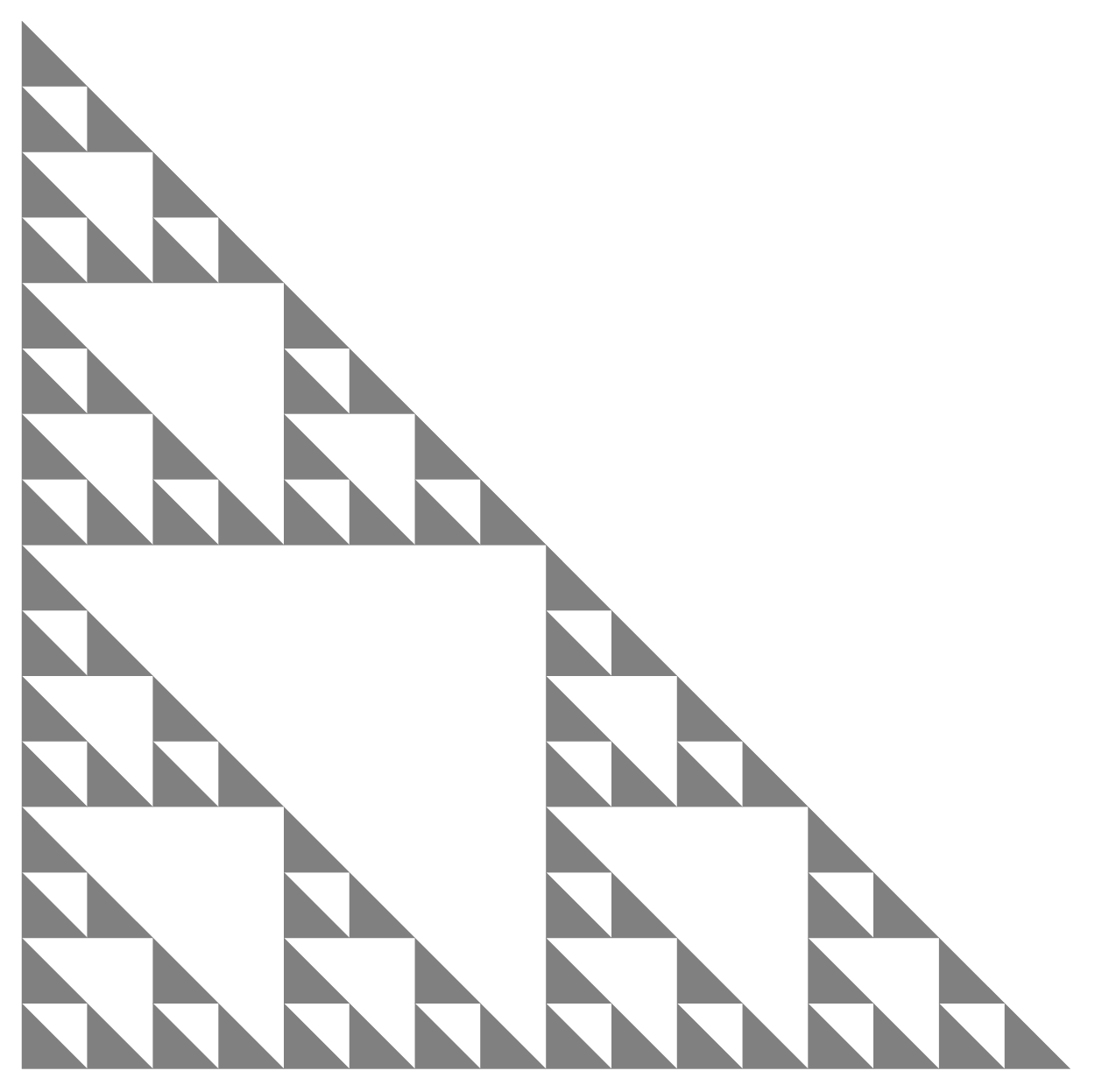} & \includegraphics[width=0.2\columnwidth]{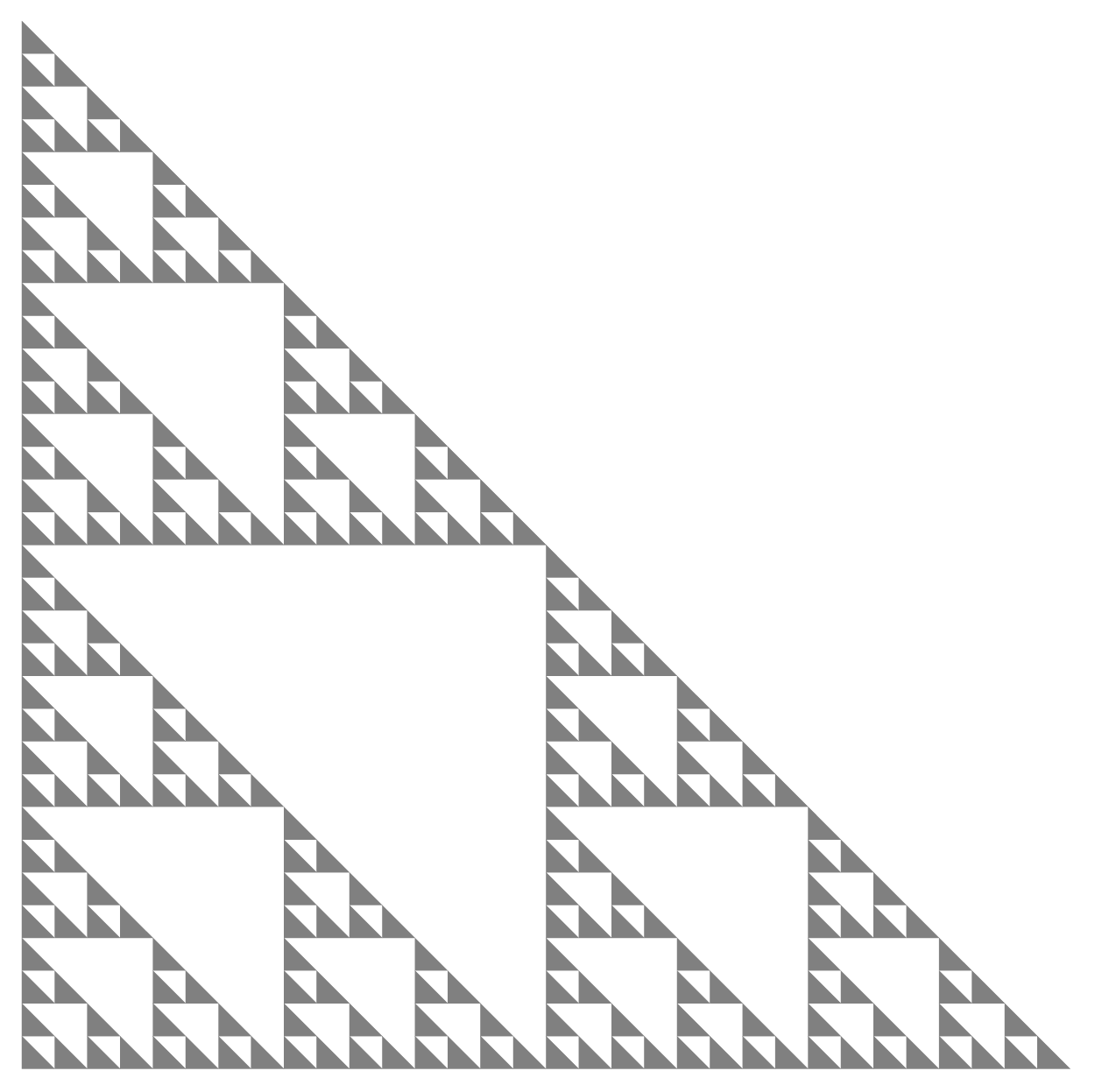}\tabularnewline
\end{tabular}

\caption{\label{fig:isp}Construction of the Sierpinski gasket.}

\end{figure}

We have the random variable $B^{\mathbb{N}}\xrightarrow{\;X\;}W$,
given by
\begin{equation}
\omega=\left(b_{i_{1}},b_{i_{2}},b_{i_{3}},\cdots\right)\longmapsto x=\sum_{k=1}^{\infty}M^{-k}b_{i_{k}}.\label{eq:F3}
\end{equation}

As a Cantor set, $W$ (the Sierpinski gasket) is the boundary of the
tree symbol representation; see \figref{wtree}.

\begin{figure}
\includegraphics[width=0.7\columnwidth]{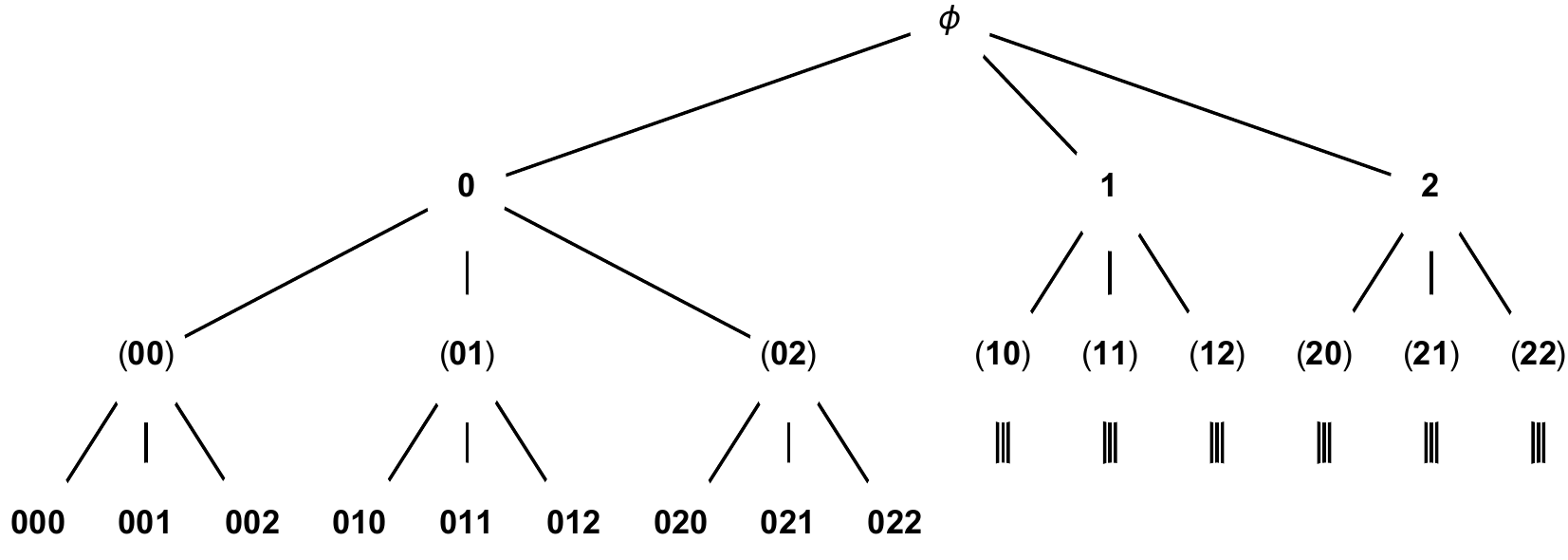}

\caption{\label{fig:wtree}Symbol representations of infinite words.}
\end{figure}

Recall that every $\omega\in B^{\mathbb{N}}$ is an infinite word
$\omega=\left(b_{i_{1}},b_{i_{2}},b_{i_{3}},\cdots\right)$, with
$i_{k}\in\left\{ 0,1,2\right\} $. Setting $\omega\big|_{n}=\left(b_{i_{1}},\cdots,b_{i_{n}}\right)$,
a finite truncated word, and $\tau_{\omega|_{n}}=\tau_{i_{n}}\circ\cdots\circ\tau_{i_{1}}$;
then $\bigcap_{n}\tau_{\omega|_{n}}\left(W\right)=\left\{ x\right\} $,
i.e., the intersection is a singleton. And we set $X\left(\omega\right)=x$.

Let $p$ be the probability distribution on $B$, where 
\begin{equation}
p=\left(\frac{1}{3},\frac{1}{3},\frac{1}{3}\right).\label{eq:F4}
\end{equation}
Let $\mathbb{P}=\vartimes_{1}^{\infty}p$, and $\mu=\mathbb{P}\circ X^{-1}$
be the corresponding IFS measure, i.e., $\mu$ is the unique Borel
probability measure on $W$, s.t. 
\begin{equation}
d\mu=\frac{1}{3}\sum_{j=0}^{2}\mu\circ\tau_{j}^{-1}.\label{eq:si3}
\end{equation}
See \secref{gifs} for details.
\begin{rem}
~
\begin{enumerate}
\item The Hausdorff dimension of $W$ is $\ln3/\ln2$, where $3=\#\left\{ B\right\} $
and $2=$ scaling number.
\item Let $O_{j}$ be the triangles removed from the $j^{th}$ iteration
(\figref{isp}), and let $O=\bigcup_{j=1}^{\infty}O_{j}$. Then, 
\[
\lambda_{2}\left(O\right)=\lambda_{2}\left(\bigcup\nolimits _{j=1}^{\infty}O_{j}\right)=\frac{1}{2}\left[\frac{1}{4}+\frac{3}{4^{2}}+\frac{3^{2}}{4^{3}}+\cdots\right]=\frac{1}{2};
\]
and so $\lambda_{2}\left(W\right)=0$, where $\lambda_{2}$ denotes
the 2D Lebesgue measure. Note that $\mu\left(W\right)=1$.
\end{enumerate}
\end{rem}

\begin{lem}
Let $W$ be the Sierpinski gasket, and $\mu$ be the corresponding
IFS measure. Let $\widehat{\mu}$ be the Fourier transform of $\mu$,
i.e., $\widehat{\mu}\left(\lambda\right):=\int_{W}e^{i2\pi\lambda\cdot x}d\mu\left(x\right)$.
Then 
\begin{equation}
\widehat{\mu}\left(\lambda\right)=\frac{1}{3}\left[1+e^{i\pi\lambda_{1}}+e^{i\pi\lambda_{2}}\right]\widehat{\mu}\left(\lambda/2\right),\label{eq:sm}
\end{equation}
where $\lambda=\left(\lambda_{1},\lambda_{2}\right)\in\mathbb{R}^{2}$.
\end{lem}

\begin{proof}
Immediate from (\ref{eq:si3}). More specifically, we have 
\begin{align*}
\widehat{\mu}\left(\lambda\right) & =\frac{1}{3}\sum_{j=0}^{2}\int_{W}e^{i2\pi\lambda\cdot\tau_{j}\left(x\right)}d\mu\left(x\right)\\
 & =\frac{1}{3}\Big(\int_{W}e^{i2\pi\lambda\cdot x/2}d\mu\left(x\right)+\int_{W}e^{i2\pi\lambda/2\cdot\left(x+\left(1,0\right)\right)}d\mu\left(x\right)\\
 & \qquad+\int_{W}e^{i2\pi\lambda/2\cdot\left(x+\left(0,1\right)\right)}d\mu\left(x\right)\Big)\\
 & =\frac{1}{3}\left(1+e^{i\pi\lambda_{1}}+e^{i\pi\lambda_{2}}\right)\widehat{\mu}\left(\lambda/2\right),
\end{align*}
which is the assertion (\ref{eq:sm}).
\end{proof}
By general theory (see \secref{SSM}), the IFS measure $\mu$ as in
(\ref{eq:si3}) has a disintegration 
\begin{equation}
d\mu=\int_{0}^{1}\sigma^{x}\left(dy\right)d\xi\left(x\right),\label{eq:si4}
\end{equation}
where 
\begin{equation}
\xi=\mu\circ\pi_{1}^{-1}\label{eq:si5}
\end{equation}
with $supp\left(\xi\right)\subset\left[0,1\right]$. Note, if $S\subset\left[0,1\right]$
is a measurable subset, then 
\begin{equation}
\xi\left(S\right)=\mu\left(\left\{ \left(x,y\right)\mid x\in S\right\} \right).\label{eq:si6}
\end{equation}

\begin{lem}
\label{lem:sW}Let $W$ be the Sierpinski gasket. Then points in $W$
are represented as random power series 
\begin{equation}
\begin{bmatrix}x\\
y
\end{bmatrix}\in W\Longleftrightarrow\left\{ \begin{matrix}x=\sum_{k=1}^{\infty}\varepsilon_{k}2^{-k}\\
y=\sum_{k=1}^{\infty}\eta_{k}2^{-k}
\end{matrix}\right.\label{eq:F9}
\end{equation}
where $\left(\varepsilon_{k}\right),\left(\eta_{k}\right)$ are defined
on $\Omega=\left\{ 0,1\right\} ^{\mathbb{N}}$, i.e., the binary probability
space.

Moreover, $\varepsilon_{k}$ is i.i.d. on $\left\{ 0,1\right\} $,
$k\in\mathbb{N}$, with distribution $\left(2/3,1/3\right)$. That
is, $Prob\left(\varepsilon_{k}=0\right)=2/3$, and $Prob\left(\varepsilon_{k}=1\right)=1/3$.
The same conclusion holds for $\eta_{k}$ as well.
\end{lem}

\begin{proof}
This follows from (\ref{eq:F3}) and (\ref{eq:F4}).

In detail, let $X:B^{\mathbb{N}}\rightarrow W$ be the random variable
from (\ref{eq:F3}), $X\left(\omega\right)=\sum_{k=1}^{\infty}M^{-k}b_{i_{k}}$,
for all $\omega\in B^{\mathbb{N}}$; then 
\begin{align*}
W\ni\begin{bmatrix}x\\
y
\end{bmatrix} & =X\left(\omega\right)=\sum_{k=1}^{\infty}\begin{bmatrix}2^{-k} & 0\\
0 & 2^{-k}
\end{bmatrix}b_{i_{k}}\\
 & =\sum_{k=1}^{\infty}\begin{bmatrix}2^{-k} & 0\\
0 & 2^{-k}
\end{bmatrix}\left\{ \begin{bmatrix}0\\
0
\end{bmatrix},\begin{bmatrix}1\\
0
\end{bmatrix},\begin{bmatrix}0\\
1
\end{bmatrix}\right\} \\
 & =\begin{bmatrix}\sum_{k=1}^{\infty}2^{-k}\varepsilon_{k}\left(x\right)\\
\sum_{k=1}^{\infty}2^{-k}\eta_{k}\left(x\right)
\end{bmatrix},
\end{align*}
where 
\begin{align*}
Pr\left(\varepsilon_{k}=0\right) & =Pr\left(\eta_{k}=0\right)=2/3,\\
Pr\left(\varepsilon_{k}=1\right) & =Pr\left(\eta_{k}=1\right)=1/3.
\end{align*}
Also we have the following conditional probabilities: 
\begin{align*}
Pr\left(\eta_{k}=0\mid\varepsilon_{k}=0\right) & =1/2,\\
Pr\left(\eta_{k}=1\mid\varepsilon_{k}=0\right) & =1/2,\\
Pr\left(\eta_{k}=0\mid\varepsilon_{k}=1\right) & =1.
\end{align*}
One checks that 
\begin{align}
Pr\left(\eta_{k}=0\right) & =Pr\left(\eta_{k}=0\mid\varepsilon_{k}=0\right)Pr\left(\varepsilon_{k}=0\right)\nonumber \\
 & \qquad+Pr\left(\eta_{k}=0\mid\varepsilon_{k}=1\right)Pr\left(\varepsilon_{k}=1\right)=\frac{1}{2}\cdot\frac{2}{3}+1\cdot\frac{1}{3}=\frac{2}{3},\label{eq:F10}\\
Pr\left(\eta_{k}=1\right) & =Pr\left(\eta_{k}=1\mid\varepsilon_{k}=0\right)Pr\left(\varepsilon_{k}=0\right)\nonumber \\
 & \qquad+Pr\left(\eta_{k}=1\mid\varepsilon_{k}=1\right)Pr\left(\varepsilon_{k}=1\right)=\frac{1}{2}\cdot\frac{2}{3}+0\cdot\frac{1}{3}=\frac{1}{3}.\label{eq:F11}
\end{align}
See the diagram in \figref{tp}.
\end{proof}
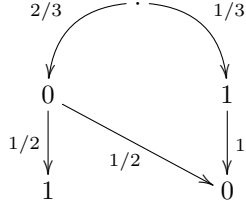
\begin{figure}
\[
\xymatrix{ & \cdot\ar@/^{1pc}/[rd]^{1/3}\ar@/_{1pc}/[ld]_{2/3}\\
0\ar[d]_{1/2}\ar[rrd]_{1/2} &  & 1\ar[d]^{1}\\
1 &  & 0
}
\]

\caption{\label{fig:tp}transition probabilities}

\end{figure}

\begin{lem}
\label{lem:sW2}Let $\mu$ be the IFS measure of the Sierpinski gasket
as above, and $\xi=\mu\circ\pi_{1}^{-1}$ be as in (\ref{eq:si4})--(\ref{eq:si6}),
so that $\mu$ has the disintegration in (\ref{eq:si4}).
\begin{enumerate}
\item Then the measure $\xi$ is singular and non-atomic. More precisely,
$\xi$ is the product measure $\vartimes_{1}^{\infty}\left\{ 2/3,1/3\right\} $
defined on $\left\{ 0,1\right\} ^{\mathbb{N}}$.
\item For a.a. $x$ w.r.t $\xi$, the measure $\sigma^{x}\left(dy\right)$
(in the $y$-variable) is singular. Hence $\mu$ is slice singular
(see \defref{SS}), and $\left\{ e_{n}\right\} _{n\in\mathbb{N}_{0}^{2}}$
is total in $L^{2}\left(\mu\right)$.
\end{enumerate}
\end{lem}

\begin{proof}
For all points $\left(x,y\right)\in W$, let $x=\sum_{k=1}^{\infty}\varepsilon_{k}2^{-k}$,
$y=\sum_{k=1}^{\infty}\eta_{k}2^{-k}$ be as in (\ref{eq:F9}). Then
(\ref{eq:F10}) \& (\ref{eq:F11}) hold for all $x\in I$.

Therefore, we get the product measure $\xi=\vartimes_{1}^{\infty}\left\{ 2/3,1/3\right\} $
on the space $\Omega=\vartimes_{1}^{\infty}\left\{ 0,1\right\} $;
see \figref{spt}. By contrast, $\lambda=\vartimes_{1}^{\infty}\left\{ 1/2,1/2\right\} $
is Lebesgue measure; hence $\xi$ and $\lambda$ are mutually singular
by Kakutani's thoerem. (See \corref{kuk} above.)

Note that, for a.a. $x$, the measure $\sigma^{x}\left(dy\right)$
is the middle interval gap supported on $A\left(x\right)=\left\{ y\mid\left(x,y\right)\in W\right\} $,
and we conclude that $\sigma^{x}\left(dy\right)$ is singular w.r.t.
Lebesgue measure for a.a. $x$. By \thmref{SM}, it follows that $\left\{ e_{n}\right\} _{n\in\mathbb{N}_{0}^{2}}$
is total in $L^{2}\left(\mu\right)$.
\end{proof}
\begin{figure}
\begin{tabular}{cccc}
\includegraphics[width=0.2\columnwidth]{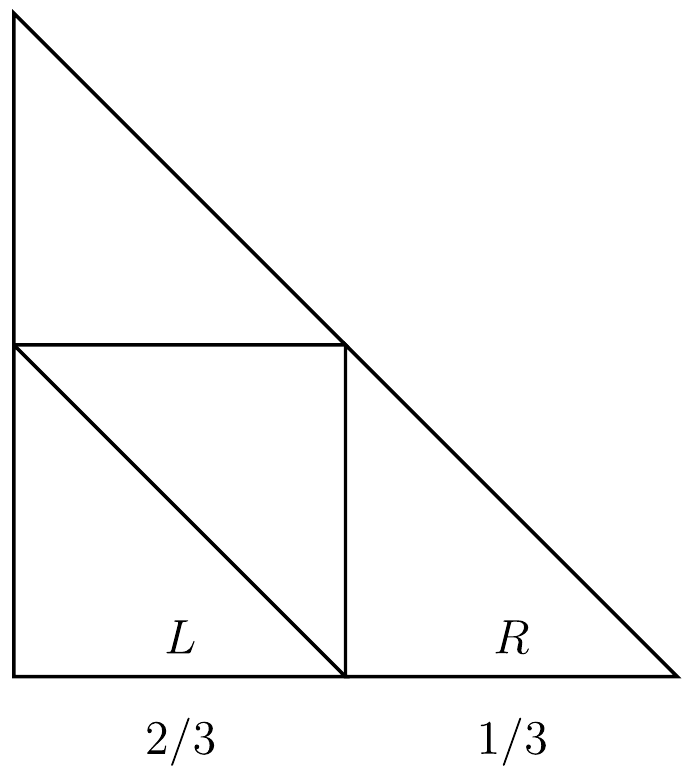} & \includegraphics[width=0.2\columnwidth]{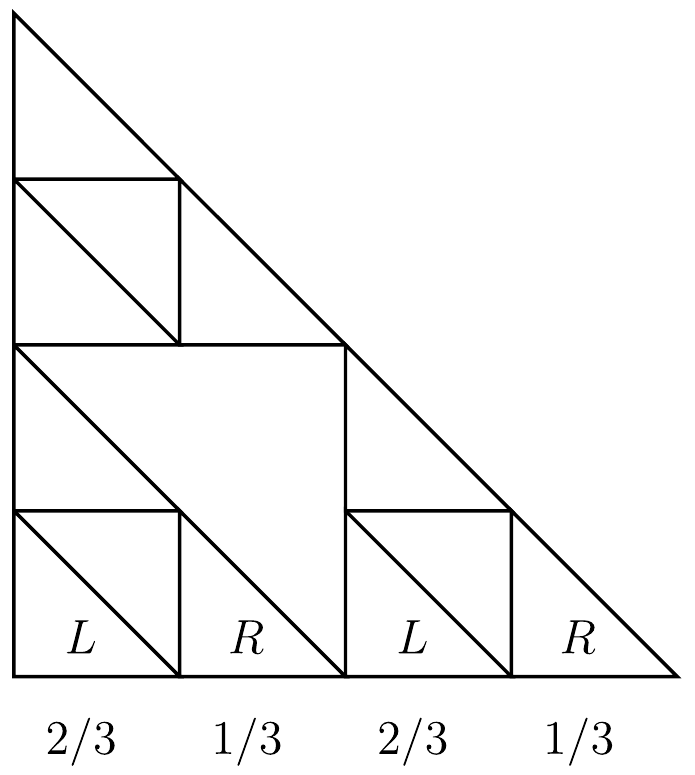} & \includegraphics[width=0.2\columnwidth]{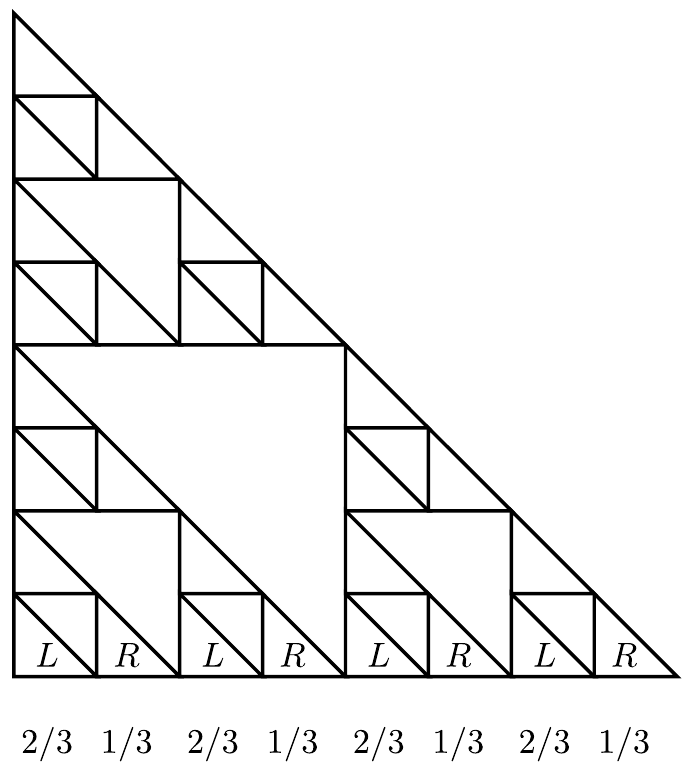} & \includegraphics[width=0.2\columnwidth]{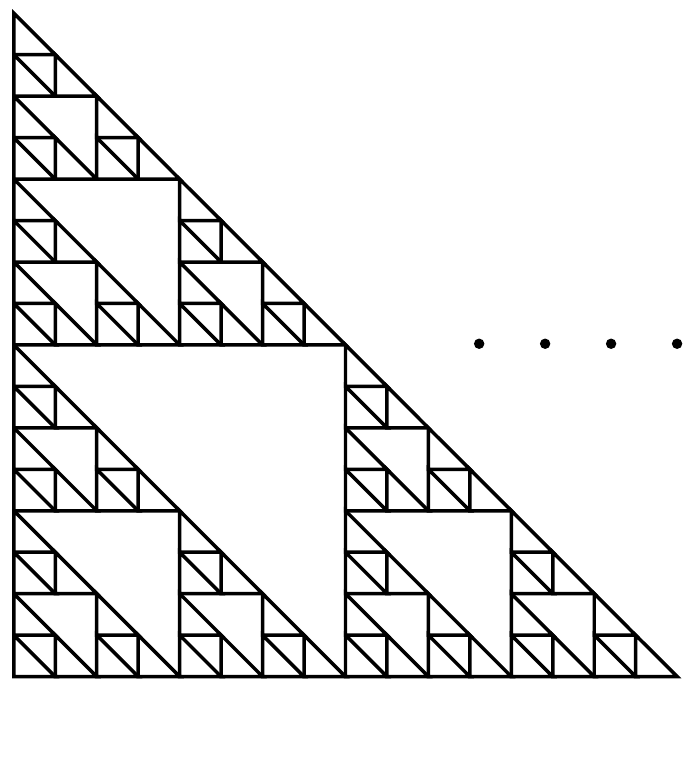}\tabularnewline
step 1 & step 2 & step 3 & step 4\tabularnewline
 &  &  & \tabularnewline
\multicolumn{4}{c}{$\xi=\mu\circ\pi_{1}^{-1}=\vartimes_{1}^{\infty}\left\{ 2/3,1/3\right\} $}\tabularnewline
\multicolumn{4}{c}{$Pr\left(\varepsilon_{k}\in L\right)=2/3,\quad Pr\left(\varepsilon_{k}\in R\right)=1/3$}\tabularnewline
\end{tabular}

\caption{\label{fig:spt}The measure $\xi$, or $d\xi\left(x\right)$ as an
infinite product measure.}
\end{figure}

\begin{rem}
\label{rem:mca}There is a Markov chain associated with the transition
probabilities (see \figref{tp}). Note that 
\[
\begin{bmatrix}2/3 & 1/3\end{bmatrix}\begin{bmatrix}1/2 & 1/2\\
1 & 0
\end{bmatrix}=\begin{bmatrix}2/3 & 1/3\end{bmatrix},
\]
so the conditional expectation can be expressed as a Perron-Frobenius
problem with the row vector $\begin{bmatrix}2/3 & 1/3\end{bmatrix}$
as a left Perron-Frobenius vector.

As another example, consider the fractal Eiffel Tower $W_{Ei}$ (see
\figref{fei}). In this case, we have 
\[
M=\begin{bmatrix}2 & 0 & 0\\
0 & 2 & 0\\
0 & 0 & 2
\end{bmatrix},\quad B=\left\{ \begin{bmatrix}0\\
0\\
0
\end{bmatrix},\begin{bmatrix}1\\
0\\
0
\end{bmatrix},\begin{bmatrix}0\\
1\\
0
\end{bmatrix},\begin{bmatrix}0\\
0\\
1
\end{bmatrix}\right\} ,
\]
and $p=\left(1/4,1/4,1/4,1/4\right)$. It follows that each coordinate
of points in $W_{Ei}$ has representation $\sum_{k=1}^{\infty}\varepsilon_{k}2^{-k}$,
where $\left\{ \varepsilon_{k}\right\} $ is i.i.d. with $Pr\left(\varepsilon_{k}=0\right)=3/4$,
and $Pr\left(\varepsilon_{k}=1\right)=1/4$. The transition probabilities
are given by the diagram below. 
\[
\xymatrix{ & \cdot\ar@/_{1pc}/[dl]_{3/4}\ar@/^{1pc}/[dr]^{1/4}\\
0\ar[d]_{1/3}\ar[rrd]^{2/3} &  & 1\ar[d]^{1}\\
1 &  & 0
}
\]
One checks that 
\[
\begin{bmatrix}3/4 & 1/4\end{bmatrix}\begin{bmatrix}2/3 & 1/3\\
1 & 0
\end{bmatrix}=\begin{bmatrix}3/4 & 1/4\end{bmatrix}.
\]
\end{rem}

\begin{figure}
\begin{tabular}{cccc}
\includegraphics[width=0.2\columnwidth]{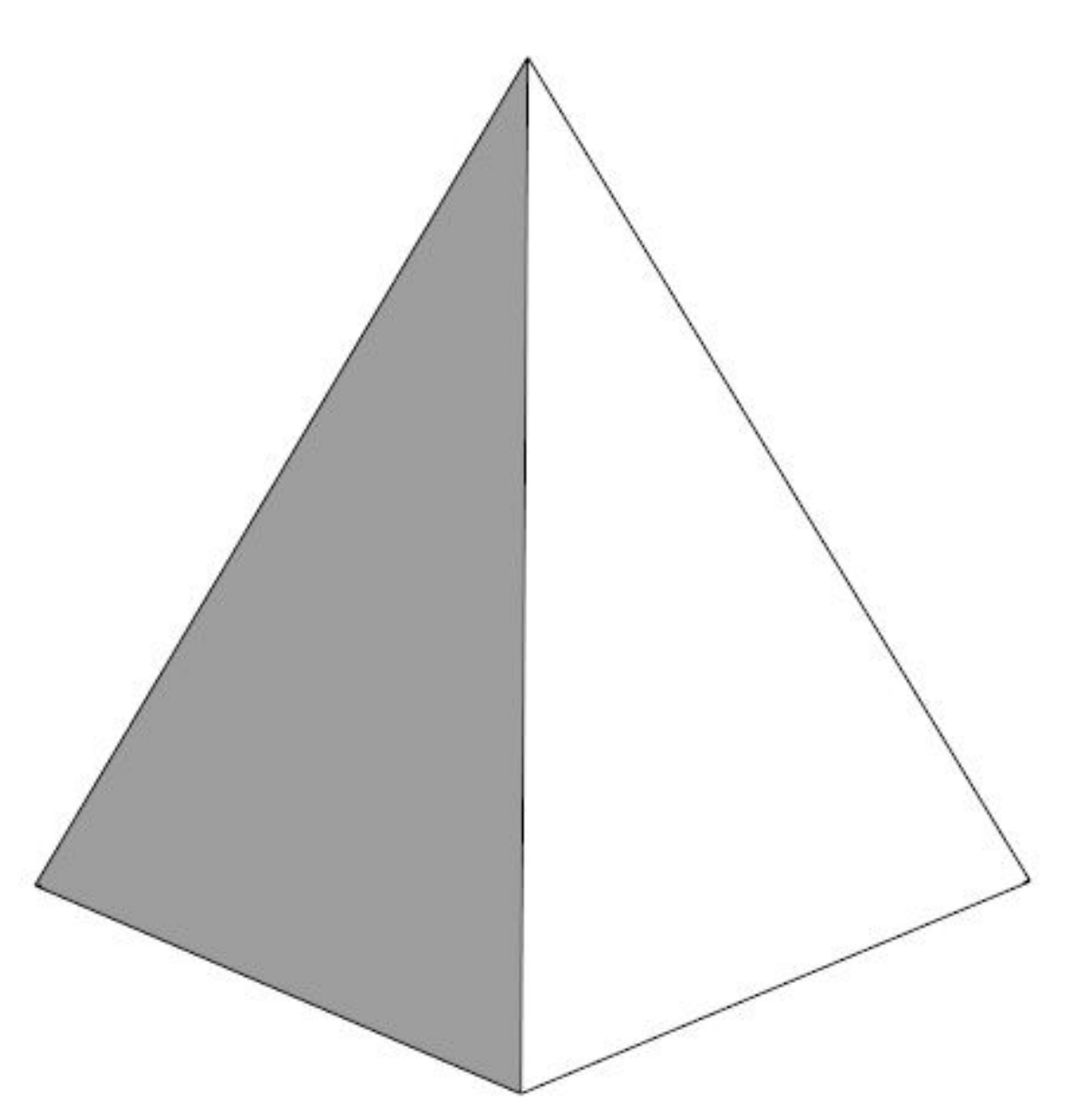} & \includegraphics[width=0.2\columnwidth]{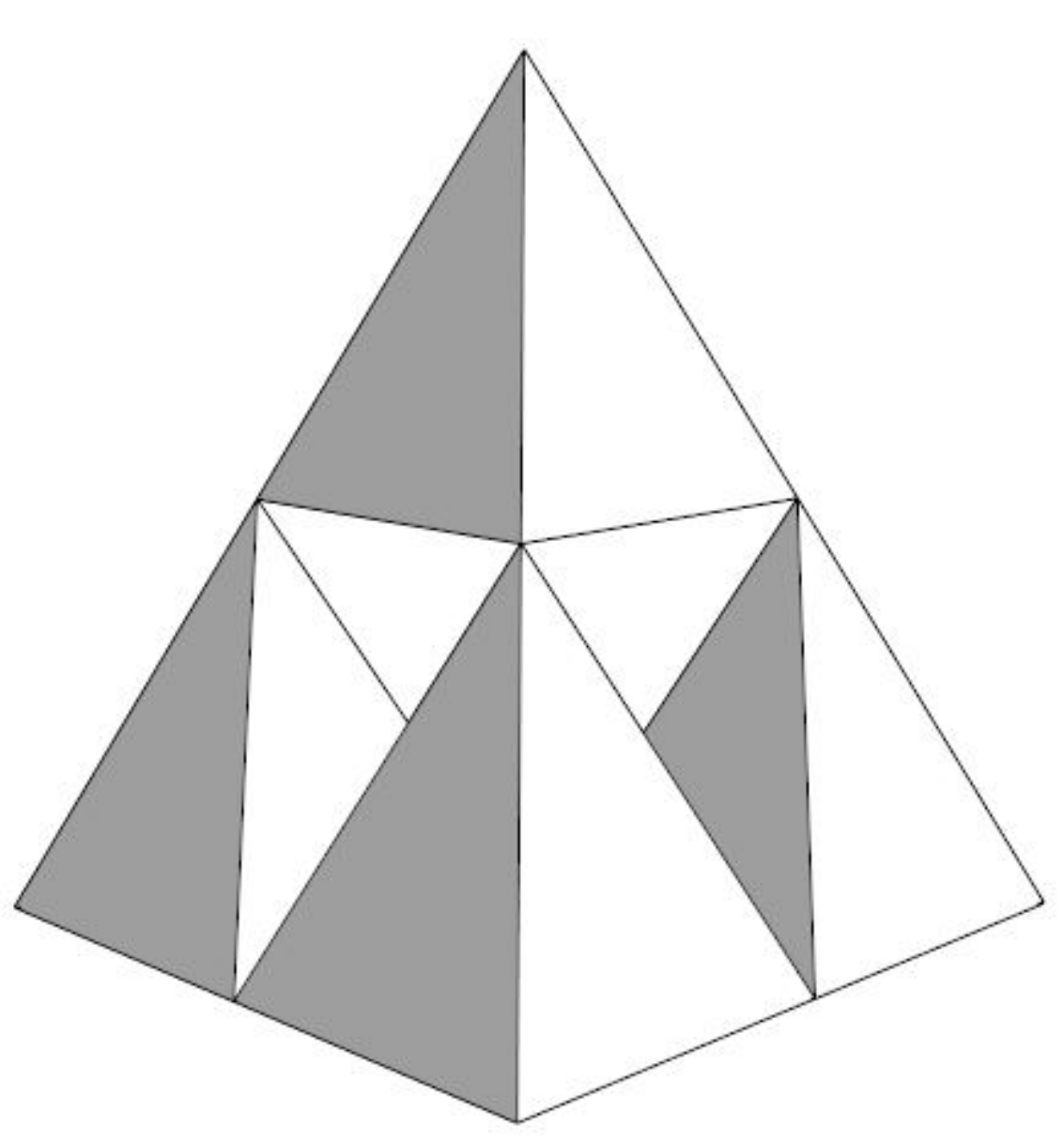} & \includegraphics[width=0.2\columnwidth]{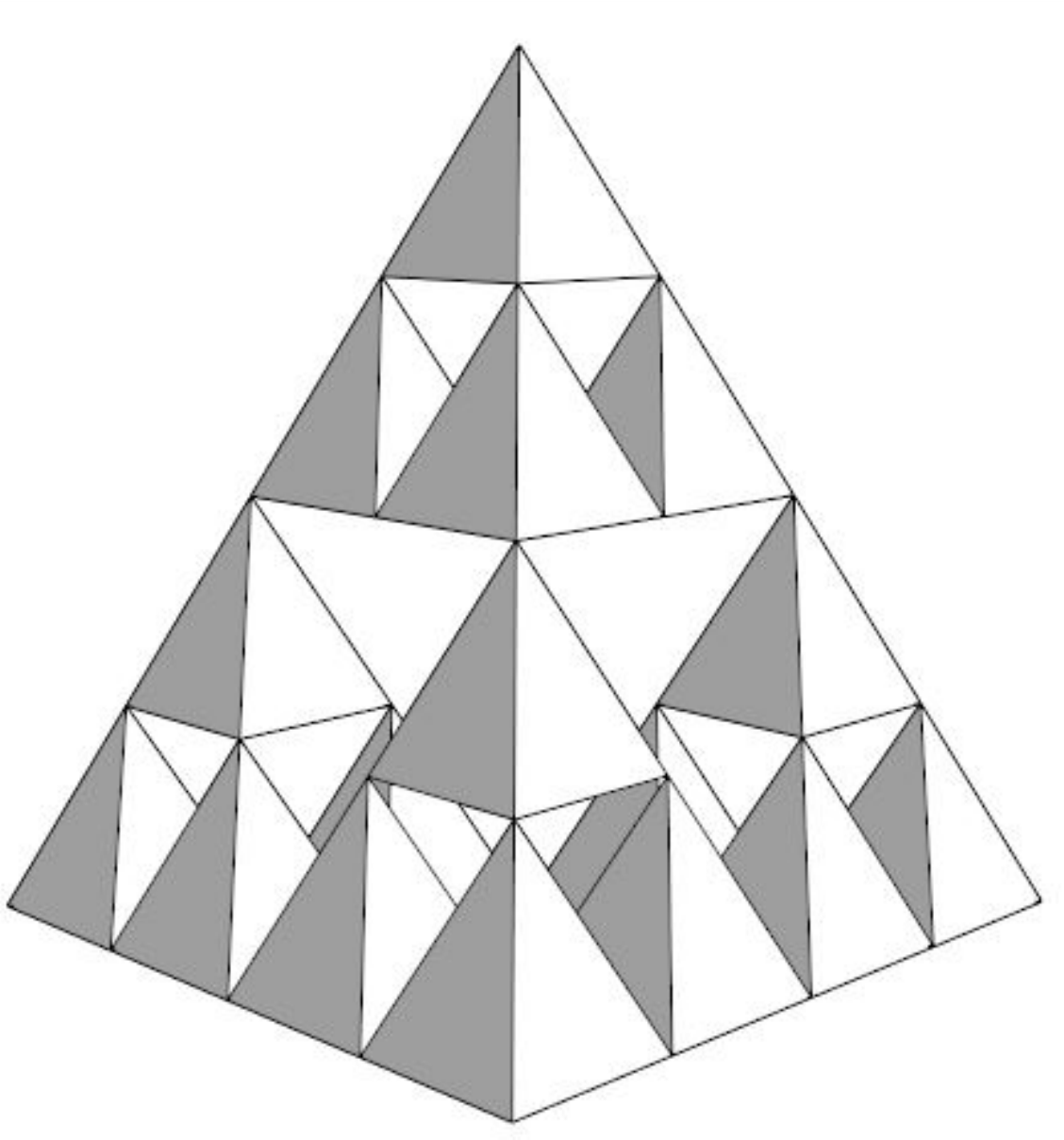} & \includegraphics[width=0.2\columnwidth]{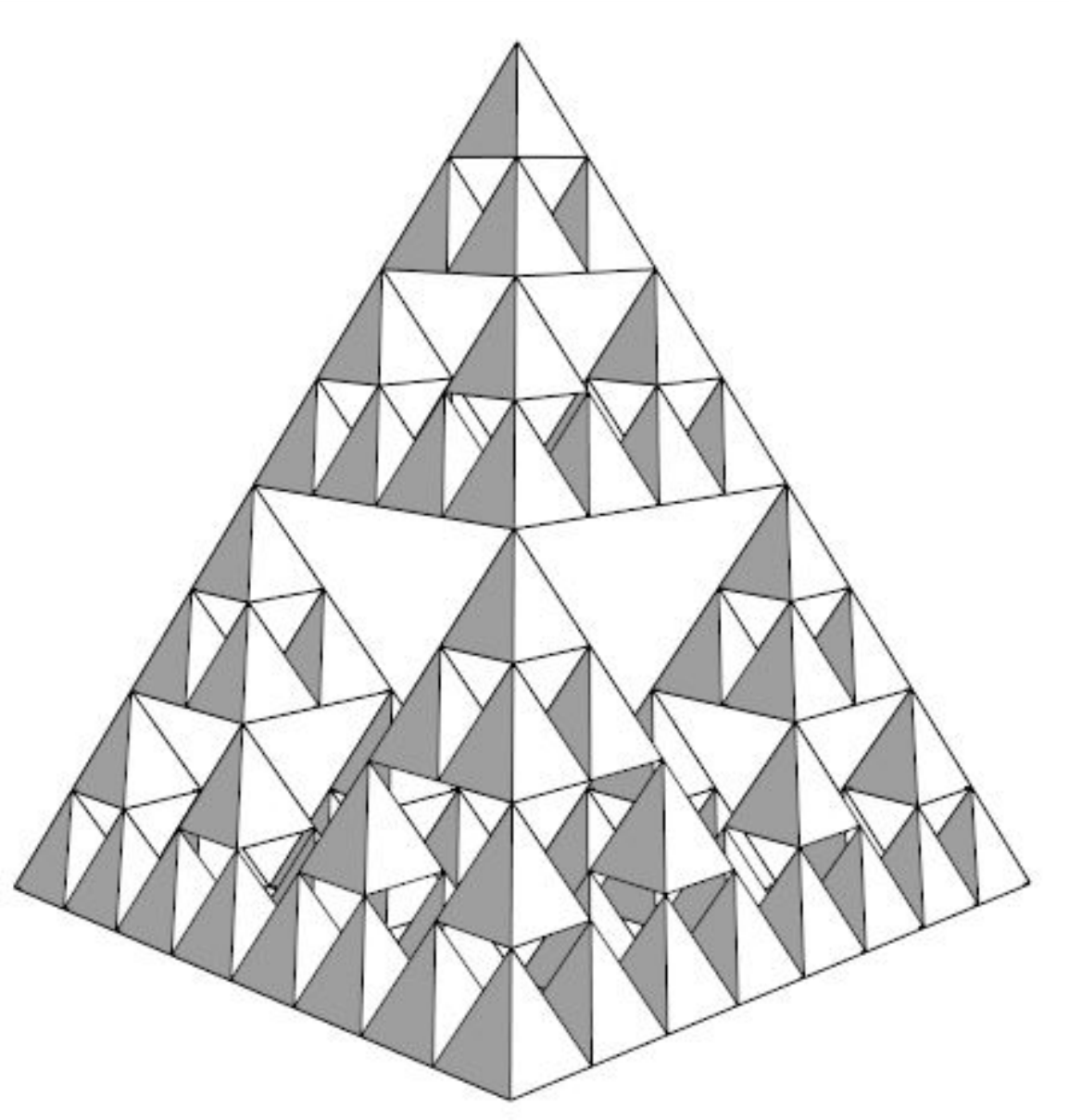}\tabularnewline
step 0 & step 1 & step 2 & step 3\tabularnewline
 &  &  & \tabularnewline
\multicolumn{4}{c}{$\xi=\mu\circ\pi_{1}^{-1}=\vartimes_{1}^{\infty}\left\{ 3/4,1/4\right\} $}\tabularnewline
\multicolumn{4}{c}{$Pr\left(\varepsilon_{k}=0\right)=3/4,\quad Pr\left(\varepsilon_{k}=1\right)=1/4$}\tabularnewline
\end{tabular}

\caption{\label{fig:fei}Construction of the fractal Eiffel Tower.}
\end{figure}

\begin{conjecture}
Given an affine contractive IFS measure $\mu$ supported in $\left[0,1\right]^{d}$,
let $T=\left(T_{ij}\right)$ be the corresponding Markov transition
matrix. Then the following are equivalent:
\begin{enumerate}
\item The Fourier frequencies $\left\{ e_{n}\right\} _{n\in\mathbb{N}_{0}^{d}}$
are total in $L^{2}\left(\mu\right)$.
\item The Perron-Frobenius vector $v$ ($vT=v$, or $\sum_{j}v_{j}T_{ji}=v_{i}$)
is non-constant, i.e., not proportional to $\left(1,1,\cdots,1\right)$.
\end{enumerate}
\end{conjecture}

\begin{rem}[The Sierpinski carpet]
 In the above, we carried out all the detailed computation justifying
our conclusions for the case of the Sierpinski gasket, \figref{2df}
(A). Recall that \figref{2df} (B) represents the Sierpinski carpet,
a close cousin; and the reader will be able to fill in details from
inside the section, spelling out the changes from (A) to (B). In case
(B), naturally, the particular affine transformations (\ref{eq:F1})--(\ref{eq:F2})
are a bit different (i.e., for case (B)), but they are of the same
nature. In particular, it follows that the maximal entropy (IFS) measure
for the Sierpinski carpet is also slice-singular. Moreover, the other
conclusions from Lemmas \ref{lem:sW}, and \ref{lem:sW2}, and \remref{mca},
carry over from case (A) to case (B), \emph{mutatis mutandis}. As
the underlying ideas and methods involved are the same, interested
readers will be able to fill in details.

Moreover the above remarks, regarding extension of the conclusions
for case (A) to that of (B), also apply \emph{mutatis mutandis}, to
the case of \figref{fei}, the fractal Eiffel Tower. There again,
we conclude that the associated maximal entropy (IFS) measure is also
slice-singular.
\end{rem}

\begin{acknowledgement*}
The co-authors thank the following colleagues for helpful and enlightening
discussions: Professors Daniel Alpay, Sergii Bezuglyi, Ilwoo Cho,
Wayne Polyzou, Eric S. Weber, and members in the Math Physics seminar
at The University of Iowa.

\bibliographystyle{amsalpha}
\bibliography{ref}
\end{acknowledgement*}

\end{document}